\documentclass[12pt]{amsart}
\usepackage{amssymb}
\begin{document}
\numberwithin{equation}{section}

\def\1#1{\overline{#1}}
\def\2#1{\widetilde{#1}}
\def\3#1{\widehat{#1}}
\def\4#1{\mathbb{#1}}
\def\5#1{\frak{#1}}
\def\6#1{{\mathcal{#1}}}

\newcommand{\UH}{\mathbb{H}}
\newcommand{\de}{\partial}
\newcommand{\R}{\mathbb R}
\newcommand{\Ha}{\mathbb H}
\newcommand{\al}{\alpha}
\newcommand{\tr}{\widetilde{\rho}}
\newcommand{\tz}{\widetilde{\zeta}}
\newcommand{\tk}{\widetilde{C}}
\newcommand{\tv}{\widetilde{\varphi}}
\newcommand{\hv}{\hat{\varphi}}
\newcommand{\tu}{\tilde{u}}
\newcommand{\tF}{\tilde{F}}
\newcommand{\debar}{\overline{\de}}
\newcommand{\Z}{\mathbb Z}
\newcommand{\C}{\mathbb C}
\newcommand{\Po}{\mathbb P}
\newcommand{\zbar}{\overline{z}}
\newcommand{\G}{\mathcal{G}}
\newcommand{\So}{\mathcal{S}}
\newcommand{\Ko}{\mathcal{K}}
\newcommand{\U}{\mathcal{U}}
\newcommand{\B}{\mathbb B}
\newcommand{\oB}{\overline{\mathbb B}}
\newcommand{\Cur}{\mathcal D}
\newcommand{\Dis}{\mathcal Dis}
\newcommand{\Levi}{\mathcal L}
\newcommand{\SP}{\mathcal SP}
\newcommand{\Sp}{\mathcal Q}
\newcommand{\A}{\mathcal O^{k+\alpha}(\overline{\mathbb D},\C^n)}
\newcommand{\CA}{\mathcal C^{k+\alpha}(\de{\mathbb D},\C^n)}
\newcommand{\Ma}{\mathcal M}
\newcommand{\Ac}{\mathcal O^{k+\alpha}(\overline{\mathbb D},\C^{n}\times\C^{n-1})}
\newcommand{\Acc}{\mathcal O^{k-1+\alpha}(\overline{\mathbb D},\C)}
\newcommand{\Acr}{\mathcal O^{k+\alpha}(\overline{\mathbb D},\R^{n})}
\newcommand{\Co}{\mathcal C}
\newcommand{\Hol}{{\sf Hol}}
\newcommand{\Aut}{{\sf Aut}(\mathbb D)}
\newcommand{\D}{\mathbb D}
\newcommand{\oD}{\overline{\mathbb D}}
\newcommand{\oX}{\overline{X}}
\newcommand{\loc}{L^1_{\rm{loc}}}
\newcommand{\la}{\langle}
\newcommand{\ra}{\rangle}
\newcommand{\thh}{\tilde{h}}
\newcommand{\N}{\mathbb N}
\newcommand{\kd}{\kappa_D}
\newcommand{\Hr}{\mathbb H}
\newcommand{\ps}{{\sf Psh}}
\newcommand{\Hess}{{\sf Hess}}
\newcommand{\subh}{{\sf subh}}
\newcommand{\harm}{{\sf harm}}
\newcommand{\ph}{{\sf Ph}}
\newcommand{\tl}{\tilde{\lambda}}
\newcommand{\gdot}{\stackrel{\cdot}{g}}
\newcommand{\gddot}{\stackrel{\cdot\cdot}{g}}
\newcommand{\fdot}{\stackrel{\cdot}{f}}
\newcommand{\fddot}{\stackrel{\cdot\cdot}{f}}

\def\Re{{\sf Re}\,}
\def\Im{{\sf Im}\,}

\newcommand{\Real}{\mathbb{R}}
\newcommand{\Natural}{\mathbb{N}}
\newcommand{\Complex}{\mathbb{C}}
\newcommand{\ComplexE}{\overline{\mathbb{C}}}
\newcommand{\Int}{\mathbb{Z}}
\newcommand{\UD}{\mathbb{D}}
\newcommand{\clS}{\mathcal{S}}
\newcommand{\gtz}{\ge0}
\newcommand{\gt}{\ge}
\newcommand{\lt}{\le}
\newcommand{\fami}[1]{(#1_{s,t})}
\newcommand{\famc}[1]{(#1_t)}
\newcommand{\ts}{t\gt s\gtz}
\newcommand{\classCC}{\tilde{\mathcal C}}
\newcommand{\classS}{\mathcal S}

\newcommand{\Step}[2]{\begin{itemize}\item[{\bf Step~#1.}]{\it #2}\end{itemize}}
\newcommand{\step}[2]{\begin{itemize}\item[{\it Step~#1.}]{\it #2}\end{itemize}}
\newcommand{\proofbox}{\hfill$\Box$}

\newcommand{\mcite}[1]{\csname b@#1\endcsname}
\newcommand{\UC}{\mathbb{T}}

\newcommand{\Moeb}{\mathrm{M\ddot ob}}

\newcommand{\dAlg}{{\mathcal A}(\UD)}
\newcommand{\diam}{\mathrm{diam}}

\theoremstyle{theorem}
\newtheorem {result} {Theorem}
\setcounter {result} {64}
 \renewcommand{\theresult}{\char\arabic{result}}



\newcommand{\Spec}{\Lambda^d}
\newcommand{\SpecR}{\Lambda^d_R}
\newcommand{\Prend}{\mathrm P}




\def\cn{{\C^n}}
\def\cnn{{\C^{n'}}}
\def\ocn{\2{\C^n}}
\def\ocnn{\2{\C^{n'}}}
\def\je{{\6J}}
\def\jep{{\6J}_{p,p'}}
\def\th{\tilde{h}}


\def\dist{{\rm dist}}
\def\const{{\rm const}}
\def\rk{{\rm rank\,}}
\def\id{{\sf id}}
\def\aut{{\sf aut}}
\def\Aut{{\sf Aut}}
\def\CR{{\rm CR}}
\def\GL{{\sf GL}}
\def\Re{{\sf Re}\,}
\def\Im{{\sf Im}\,}
\def\U{{\sf U}}

\def\la{\langle}
\def\ra{\rangle}

\def\csub{\subset\subset}

\newcommand{\sgn}{\mathop{\mathrm{sgn}}}

\emergencystretch15pt \frenchspacing

\newtheorem{theorem}{Theorem}[section]
\newtheorem{lemma}[theorem]{Lemma}
\newtheorem{proposition}[theorem]{Proposition}
\newtheorem{corollary}[theorem]{Corollary}

\theoremstyle{definition}
\newtheorem{definition}[theorem]{Definition}
\newtheorem{example}[theorem]{Example}

\theoremstyle{remark}
\newtheorem{remark}[theorem]{Remark}
\numberwithin{equation}{section}

\newenvironment{mylist}{\begin{list}{}%
{\labelwidth=2em\leftmargin=\labelwidth\itemsep=.4ex plus.1ex
minus.1ex\topsep=.7ex plus.3ex
minus.2ex}%
\let\itm=\item\def\item[##1]{\itm[{\rm ##1}]}}{\end{list}}

\title[Local duality in Loewner equations.]{Local duality in Loewner equations.}

\author[M. D. Contreras]{Manuel D. Contreras\,$^\dag$}

\author[S. D\'{\i}az-Madrigal]{Santiago D\'{\i}az-Madrigal\,$^\ddag$}
\address{Camino de los Descubrimientos, s/n\\
Departamento de Matem\'{a}tica Aplicada II\\
Escuela T\'{e}cnica Superior de Ingenier\'{\i}a\\
Universidad de Sevilla\\
Sevilla, 41092\\
Spain.}\email{contreras@us.es} \email{madrigal@us.es}

\author[P. Gumenyuk]{Pavel Gumenyuk\,$^\sharp$}
\address{Dipartimento di Matematica, Universit\`a di Roma ``Tor Vergata", Via della Ricerca
Scientifica 1, 00133, Roma, Italia.} \email{gumenyuk@mat.uniroma2.it}

\date{\today }


\thanks{$^\dag{~}^\ddag{~}^\sharp$ Partially supported  by the ESF
Networking Programme ``Harmonic and Complex Analysis and its Applications'' and
by \textit{La Consejer\'{\i}a Econom\'\i a, Innovaci\'on y Ciencia de la Junta
de Andaluc\'\i a} (research group FQM-133).}

\thanks{$^\dag{~}^\ddag$ Partially supported by \textit{Ministerio
de Ciencia e Innovaci\'on} and the European Union (FEDER), project
MTM2009-14694-C02-02.}

\thanks{$^\ddag{~}^\sharp$ Partially supported by  the {\it Institut Mittag-Leffler }(Djursholm, Sweden).}

\thanks{$^\sharp$ Partially supported by Progetto FIRB -- Futuro in Ricerca 2008 n.\,RBFR08B2HY
{\it''Geometria Differenziale Complessa e Dinamica Olomorfa"}.}

\begin{abstract}
Among diversity of frameworks and constructions introduced in Loewner Theory by
different authors, one can distinguish two closely related but still different
ways of reasoning, which colloquially may be described as ``increasing" and
``decreasing". In this paper we review in short the main types of
(deterministic) Loewner evolution discussed in the literature and describe in
detail the local duality between ``increasing" and ``decreasing" cases  within
the general unifying approach in Loewner Theory proposed recently
in~\cite{BCM1,BCM2,SMP}. In particular, we extend several results
of~\cite{Bauer}, which deals with the chordal Loewner evolution, to this
general setting. Although the duality is given by a simple change of the
parameter, not all the results for the ``decreasing" case can be obtained by
mere translating the corresponding results for the ``increasing" case. In
particular, as a byproduct of establishing local duality between evolution
families and their ``decreasing" counterparts we obtain a new characterization
of generalized Loewner chains.
\end{abstract}

\maketitle

\tableofcontents

\def\Gen{\mathcal G}
\def\Gena{\Gen_{\mathrm{Aut}}}
\def\RH{\mathbb C_+}
\def\mes{\mathop{\mathrm{mes}}}

\newcommand{\REM}[1]{\relax}

\section{Introduction}

\subsection{Loewner Theory in different frameworks}\label{SS_1}
Last decades there has been a great burst of interest to Loewner Theory both in
its deterministic variant, going back to Charles Loewner's seminal
paper~\cite{Loewner} of 1923, and more recent stochastic version by Odded
Schramm~\cite{Schramm}.

Motivated by extremal problems for univalent functions, Loewner~\cite{Loewner} considered
univalent holomorphic functions $f(z)=z+a_2z^2+\ldots$ mapping the unit
disk~$\UD:=\{z:|z|<1\}$ onto the complex plane~$\Complex$ minus a Jordan arc $\Gamma$
extending to infinity. Choose a (unique) parametrization $\gamma:[0,+\infty)\to\Complex$ of
the Jordan arc~$\Gamma$ satisfying the following two conditions: (1)~$\gamma(t)\to\infty$ as
$t\to+\infty$; (2) $f'_t(0)=e^t$ for each $t\ge0$, where $f_t$ stands for the conformal
mapping of~$\UD$ onto $\Omega_t:=\Complex\setminus\gamma\big([t,+\infty)\big)$ normalized by
$f_t(0)=0$, $f_t'(0)>0$. Loewner proved that for any Jordan arc $\Gamma$ with such
parametrization there exists a (unique) continuous function $k:[0,+\infty)\to\UC:=\{z:|z|=1\}$
such that the family $(f_t)$ satisfies the PDE
\begin{equation}\label{EQ_class_L-PDE}
\frac{\partial f_t(z)}{\partial t}=zf'_t(z)\frac{1+k(t)z}{1-k(t)z},
\end{equation}
 where $f_t'(z)$ stands for
the partial derivative of $f_t(z)$ w.r.t.~$z$. Actually, $(f_t)$ is the unique
solution to this PDE which is defined and univalent as a function of~$z$ for
all $t\ge0$ and all $z\in\UD$. Moreover, for all $s\ge0$,
\begin{equation}\label{EQ_class_L}
f_s=\lim_{t\to+\infty}e^t\varphi_{s,t},
\end{equation}
 where
$\varphi_{s,t}:=f_t^{-1}\circ f_s:\UD\to\UD$, $0\le s\le t$, solves the characteristic ODE
\begin{equation}\label{EQ_class_L-ODE}
\frac{d\varphi_{s,t}(z)}{dt}=-
\varphi_{s,t}(z)\frac{1+k(t)\varphi_{s,t}(z)}{1-k(t)\varphi_{s,t}(z)},~~~t\ge s,\quad
\varphi_{s,s}(z)=z,~~z\in\UD.
\end{equation}
He also proved that conversely, given a continuous
${k:[0,+\infty)\to\UC}$, equation~\eqref{EQ_class_L-ODE} and
formula~\eqref{EQ_class_L} define together, in a unique way, a
family~$(f_t)_{t\ge0}$ of holomorphic univalent functions in~$\UD$
satisfying~\eqref{EQ_class_L-PDE} and such that $f_s(\UD)\subset f_t(\UD)$
whenever $0\le s\le t$, although in this case $f_0$ need not be a conformal
mapping onto the complement of a Jordan arc\footnote{The first example of such
situation was discovered by Kufarev~\cite{Kufarev1}, who found a function $k$
in \eqref{EQ_class_L-ODE} for which $f_0$ maps $\UD$ onto a half-plane.}. The
differential equations~\eqref{EQ_class_L-PDE} and \eqref{EQ_class_L-ODE} are known
nowadays as the {\it Loewner PDE} and the {\it Loewner ODE}, respectively.

Later Kufarev~\cite{Kufarev} and Pommerenke~\cite{Pommerenke-65, Pommerenke}
extended Loewner's results for a more general class of families~$(f_t)$ by
replacing the Schwarz kernel in the Loewner differential equations with an
arbitrary holomorphic function of positive real part. In the theory that they
constructed, a {\it (classical) Loewner chain} is a family $(f_t)_{t\ge0}$ of
univalent holomorphic functions $f_t:\UD\to\Complex$ subject to the inclusion
condition $f_s(\UD)\subset f_t(\UD)$ whenever ${0\le s\le t}$ and the
normalization $f_t(z)=e^{t}(z+a_2(t)z^2+\ldots)$ for all ${t\ge0}$ and all
${z\in\UD}$. A {\it driving term} or a {\it classical Herglotz function} is a
function $p:\UD\times[0,+\infty)\to\{z:\Re z>0\}$ such that $p(z,\cdot)$ is
measurable in~$[0,+\infty)$ for all $z\in\UD$ and $p(\cdot,t)$ is holomorphic
in~$\UD$ with $p(0,t)=1$ for all~$t\ge0$. It is known that for any classical
Loewner chain~$(f_t)$ there exists a classical Herglotz function~$p$, unique up
to a null-set on the $t$-axis, such that $f_t$ satisfies the {\it (classical)
Loewner\,--\,Kufarev PDE} $\partial f_t(z)/\partial t=zf_t'p(z,t)$. In turn, the
functions $\varphi_{s,t}:=f_t^{-1}\circ f_s$, $0\le s\le t$, satisfy the
{\it Loewner\,--\,Kufarev} ODE
$(d/dt)\varphi_{s,t}(z)=-\varphi_{s,t}(z)p\big(\varphi_{s,t}(z),t\big)$ with
the initial condition $\varphi_{s,s}=\id_{\UD}$, and again as in Loewner's
original theory, the family $(f_t)$ can be reconstructed
using~\eqref{EQ_class_L}. Conversely, for any classical Herglotz function~$p$
there exists a unique classical Loewner chain $(f_t)$ satisfying the
Loewner\,--\,Kufarev PDE. This Loewner chain is given by~\eqref{EQ_class_L},
where $(\varphi_{s,t})$ is the unique solution to the Loewner\,--\,Kufarev ODE
with the initial condition~$\varphi_{s,s}=\id_{\UD}$.

In modern literature the process described by the Loewner\,--\,Kufarev
differential equations is referred to as the {\it radial Loewner evolution}. It
involves holomorphic functions normalized at the origin. An analogous process
for functions normalized at a boundary point was considered by Kufarev and his
students~\cite{Kufarev_etal}, see also~\cite{Sobolev1970}. Consider the upper
half-plane $\UH:=\{z:\Im z>0\}$ with a slit $\Gamma$ along a Jordan arc.
``Slit" means here that $\Gamma$ is the image of an injective continuous
function $\gamma:[0,T]\to\UH\cup\Real$, $T>0$, such that
$\gamma\big([0,T)\big)\subset\UH$ and $\gamma(T)\in\Real$. For each $t\in[0,T]$
there exists a unique conformal mapping $f_t$ of~$\UH$ onto
$\Omega_t:=\UH\setminus\gamma\big([t,T)\big)$ normalized by its expansion at
$\infty$, $f_t(z)=z+c(t)/z+\ldots$, $c(t)<0$. Changing the parametrization of
$\Gamma$, one may assume that $c(t)=2(t-T)$ for all $t\in[0,T]$. Similarly, to
Loewner's original setting, the functions $f_t$ satisfy a PDE, and the
functions $\varphi_{s,t}:=f_t^{-1}\circ f_s$, $0\le s\le t\le T$, satisfy the
corresponding characteristic ODE. Namely,
\begin{equation}\label{EQ_chordal_PDE}
\frac{\partial f_t(z)}{\partial t}=-f_t'(z)\frac{2}{\lambda(t)-z},
\end{equation}
\begin{equation}\label{EQ_chordal_ODE}
\frac{d\varphi_{s,t}(z)}{dt}=\frac{2}{\lambda(t)-\varphi_{s,t}(z)},~~~t\in[s,T],\quad
\varphi_{s,s}(z)=z,~~z\in\UH,
\end{equation}
where $\lambda:[0,T]\to\Real$ is a continuous function determined in a unique
way by the slit~$\Gamma$. Conversely, for any such function $\lambda$ there
exists a unique family of holomorphic univalent functions $f_t:\UH\to\UH$,
$f_t(z)=z-2(T-t)/z+\ldots$ (as $z\to\infty$), $t\in[0,T]$, with
$f_s(\UH)\subset f_t(\UH)$ whenever $0\le s\le t\le T$ satisfying
PDE~\eqref{EQ_chordal_PDE}. In turn, the corresponding family
$\varphi_{s,t}:=f_t^{-1}\circ f_s$, $0\le s\le t\le T$,
satisfies~\eqref{EQ_chordal_ODE}. These two differential equations are known in
the modern literature as the {\it chordal Loewner equations}\,\footnote{Up to
our best knowledge these equations are due to P.P.Kufarev, with the chordal
Loewner ODE for the first time mentioned in his paper~\cite{Kufarev1946}
without the factor~$2$. This factor is not principal, but makes the radial and
chordal Loewner ODEs to have the same asymptotic behaviour of the vector field
near the pole on the boundary. This is convenient for the study of relationship
between the geometry of the solutions and analytic properties of the control
functions $k$ and $\lambda$.}. Since $f_T=\id_\UH$, formula~\eqref{EQ_class_L}
for the radial case is now replaced by the simpler relation $f_t=\varphi_{t,T}$
for all $t\in[0,T]$.

Another approach to the chordal Loewner equation, based on the reduction to the
radial Loewner equation by means of a time-dependent conformal change of
variable can be found in the book~\cite{Aleksandrov}, which is also a good
source for the classical radial Loewner Theory and its application to Extremal
Problems for univalent functions.

Essentially equivalent generalizations of the slit chordal Loewner evolution
were given in~\cite{Aleks1983, AleksST, Goryainov-Ba}. In this case the Cauchy
kernel in the right-hand side of \eqref{EQ_chordal_PDE} and \eqref{EQ_chordal_ODE}
is replaced by a function with positive imaginary part, locally integrable
in~$t$ and holomorphic in~$z$ with a particular regularity and normalization
condition at~$\infty$.

The complete analogy between the radial and chordal slit Loewner evolutions,
described above, is broken by the fact that the former is defined for all
$t\ge0$, while the latter is limited in time to a compact interval. Consider
the following simple construction. Reparametrize the slit $\Gamma$ in $\UH$ in
such a way that $\gamma(0)\in\Real$, $\gamma\big((0,T]\big)\subset \UH$, and
the functions $f_t$ are normalized by $f_t(z)=z-2t/z+\ldots$ This simple trick
allows to consider chordal Loewner evolution defined for all $t\ge0$. In this
case the system of domains $\Omega_t:=f_t(\UH)$ is decreasing and obtained by
removing a growing slit from~$\UH$. Accordingly, the right-hand side of
equations~\eqref{EQ_chordal_PDE} and \eqref{EQ_chordal_ODE} change the sign.
This version of the chordal Loewner evolution appeared in the well-celebrated
paper~\cite{Schramm} of 2000 by Odded Schramm\footnote{In fact he also
considered the radial case of the Loewner evolution, but the chordal case
proved to be more important for applications in the context of SLE.}, who
constructed a stochastic version of Loewner evolution (SLE) replacing the
control function~$\lambda$ with the Brownian motion times a positive
coefficient. This invention by Schramm proved to be very important in
connection with its applications in Statistical Physics, see
e.g.~\cite{LawlerBook}.

The deterministic chordal Loewner evolution underlying SLE is driven by the
following initial value problem for the Loewner chordal ODE (with the ``$-$"
sign)
\begin{equation}\label{EQ_Scr_deter}
\frac{dw}{dt}=-\frac{2}{\lambda(t)-w},~~~t\ge0,\quad w|_{t=0}=z\in\UH.
\end{equation}
The connection between the analogue of the classical Loewner chains in this
setting and the above initial value problem is given by the following
assertion, see\footnote{For simplicity we formulate here only the special case
of Theorem~4.6 from~\cite{LawlerBook}, when the family of measures $\mu_t$,
defining the vector field in the right-hand side of (a generalization of) the
chordal Loewner equation, is a family of Dirac measures. In this case the
equation in \cite[Theorem 4.6]{LawlerBook} reduces to the above chordal Loewner
equation~\eqref{EQ_Scr_deter}.}, e.g.,\,\cite[Theorem 4.6, p.\,93; Remark~4.10,
p.\,95]{LawlerBook},
\begin{result}\label{TH_LawlerBook}
Let $\lambda:[0,+\infty)\to\Real$ be a continuous function. Then the following
assertions hold:
\begin{mylist}
\item[(i)] For every $z\in\UH$, there exists a unique maximal\,\footnote{In this case
``maximal", or ``non-extendable", means that there are no solutions $[0,T)\ni
t\mapsto \tilde w_z(t)\in\UH$ to~\eqref{EQ_Scr_deter}, with $T\in(0,+\infty]$,
such that $w_z$ is the restriction of $\tilde w_z$ to a proper subset
of~$[0,T)$.} solution~$w(t)=w_z(t)\in\UH$ to initial value
problem~\eqref{EQ_Scr_deter}.
\item[(ii)] For every $t\ge0$, the set $\Omega_t$ of all $z\in\UH$ for which
$w_z$ is defined at the point\,~$t$ is a simply connected domain and the
function $g_t:\Omega_t\to\UH$ defined by $g_t(z):=w_z(t)$ for all
$z\in\Omega_t$, is the unique conformal mapping of $\Omega_t$ onto $\UH$ such
that $g_t(z)-z\to0$ as $z\to\infty$, $z\in\UH$. Moreover, for each~$t\ge0$,
$$
g_t(z)=z+\frac{2t}{z}+O(1/|z|^2)\quad\text{as $z\to\infty$, $z\in\UH$}.
$$

\item[(iii)] The family $f_t:=g_t^{-1}:\UH\to\Omega_t$, $t\ge0$,
satisfies the following initial value problem for the chordal Loewner PDE:
\begin{equation*}
\frac{\partial f_t(z)}{\partial t}=\frac{2}{\lambda(t)-z}\frac{\partial
f_t(z)}{\partial z},\qquad f_0=\id_\UH.
\end{equation*}
\end{mylist}
\end{result}

Nowadays the chordal Loewner evolution is more often considered in this
``decreasing" framework rather than in the ``increasing" framework used by
Kufarev et al~\cite{Kufarev_etal}, see e.g.~\cite{Bauer}. In particular, the
geometry of the so-called Loewner hulls $K_t:=\UH\setminus f_t(\UH)$ was
studied in a series of papers launched by~\cite{Marshall-Rohde, Lind}. One of
the basic questions addressed is under which analytical conditions on the
control function~$\lambda$ the hulls $K_t$ are formed by a growing slit along a
Jordan arc and what are the relations between regularity of~$\lambda$ and
regularity of the slit. A similar question for the radial Loewner equation was
considered, apparently for the first time, by Kufarev~\cite{Kufarev1946} in
1946, who proved that if the control function~$k$ in
equation~\eqref{EQ_class_L-ODE} has bounded derivative on a segment $[0,T]$,
then the corresponding functions $\varphi_{s,t}$ map $\UD$ onto $\UD$ minus a
slit along a $C^1$-smooth Jordan arc. Without attempting to give the complete
list we mention some recent papers on the topic~\cite{LMR, Prokhorov-Vasiliev,
Lind-Rohde, Zhora, Wong, ProkhorovPreprint}.

In a similar way the radial Loewner evolution (both associated with the
original ``slit" Loewner equation and the more general Loewner\,--\,Kufarev
equation) can be considered in the ``decreasing" way when the functions $f_t$
are self-maps of~$\UD$ with the image domains~$f_t(\UD)$ forming a decreasing
family. Another variant, appeared in the recent publications, is the so-called
{\it whole-plane} radial Loewner evolution defined for all~$t\in\Real$, see
e.g.~\cite[\S4.3]{LawlerBook}, which however can be reduced to increasing
and decreasing Loewner evolutions.

In addition to the frameworks described above, we would like to mention several
studies~\cite{Goryainov-Ba, Goryainov, Goryainov1996, Dubovikov,
Goryainov-Kudryavtseva, GoryainovTalk} devoted to the infinitesimal
representation of semigroups in $\Hol(\UD,\UD)$, which fits very well in a
general heuristic scheme containing the increasing radial and chordal Loewner
evolutions. Although independent proofs have been given for different concrete
examples, there are common ideas which it is pertinent to describe here without
going much into details. Let $\mathcal U\subset\Hol(\UD,\UD)$ be a semigroup
w.r.t. the operation of composition containing the neutral element $\id_\UD$.
Consider a continuous one-parametric subsemigroup $(\phi_t)\subset\mathcal U$,
i.e. a family~$(\phi_t)_{t\ge0}$ in $\mathcal U$ such that $\phi_0=\id_\UD$,
$\phi_{t}\circ\phi_{s}=\phi_{t+s}$ for any $s,t\ge0$, and $\phi_{t}(z)\to z$ as
$t\to+0$ for all $z\in\UD$. It is known that for each $z\in\UD$ the function
$t\mapsto \phi_t(z)$ is the solution to the initial value problem
\begin{equation}\label{EQ_autonomous}
\frac{dw}{dt}=G(w),\quad w|_{t=0}=z,
\end{equation}
with some holomorphic function $G:\UD\to\Complex$ called the {\it infinitesimal
generator} of~$(\phi_t)$. Denote the set of all infinitesimal generators of
one-parametric semigroups in~$\mathcal U$ by $\Gen_{\mathcal U}$. In contrast
to the theory of finite-dimensional Lie groups, the union of all one-parametric
semigroups in~$\mathcal U$ does not coincide with~$\mathcal U$  in general.
That is why one has to consider the non-autonomous version of
equation~\eqref{EQ_autonomous},
\begin{equation}\label{EQ_non-aut}
dw/dt=G_t(w),
\end{equation}
where $t\mapsto G_t$ is a locally integrable family in~$\Gen_{\mathcal U}$. An
additional normalization is usually imposed on the family~$(G_t)$ depending on
the specific semigroup~$\mathcal U$ under consideration. For example, take the
semigroup $\mathcal U=\mathcal
U_{0}:=\{\varphi\in\Hol(\UD,\UD):\varphi(0)=0,\varphi'(0)>0\}$. It is known
that
$$\Gen_{\mathcal U_{0}}=\Big\{z\mapsto -zp(z):p\in\Hol\big(\UD,\{\zeta:\Re\zeta>0\}\big),~\Im p(0)=0\Big\}
\cup\Big\{G(z)\equiv0\Big\}.$$ Making a suitable change of the variable~$t$ in~\eqref{EQ_non-aut} we may
assume that $G_t'(0)=1$ for all $t\ge0$. Then we obtain the (radial) Loewner\,--\,Kufarev equation
mentioned above.

The solutions $[s,+\infty)\ni t\mapsto w=w_{z,s}(t)$ to
equation~\eqref{EQ_non-aut} with the initial condition~$w(s)=z\in\UD$ are
defined for all $z\in\UD$ whenever $0\le s\le t$  and depend on~$z$
holomorphically. The family $(\varphi_{s,t})$ defined by
$\varphi_{s,t}(z)=w_{z,s}(t)$ satisfies the algebraic conditions
$\varphi_{s,s}=\id_\UD$ and $\varphi_{u,t}\circ\varphi_{s,u}=\varphi_{s,t}$ for
any $s\ge0$, any $u\ge s$, and any $t\ge u$. A family $(\varphi_{s,t})_{0\le
s\le t}\subset \mathcal U$ is called an {\it evolution family} for $\mathcal U$
if it satisfies these two algebraic conditions plus a regularity condition
in~$t$ that insures the (local) absolute continuity of $\varphi_{s,t}$
w.r.t.~$t$. This condition depends on the specific semigroup~$\mathcal U$ under
consideration and involves, in one or another way, an additive functional
$\ell=\ell_{\mathcal U}:\mathcal U\to\Real$. (``Additive" means that
$\ell(\varphi\circ\psi)=\ell(\varphi)+\ell(\psi)$ for any
$\varphi,\psi\in\mathcal U$.) For the (radial) Loewner\,--\,Kufarev equation
and for the chordal Loewner equation, the appropriate condition is
$\ell(\varphi_{s,t})=t-s$ whenever $0\le s\le t$, where for the radial case
$\ell(\varphi):=-\log\varphi'(0)$ and for the chordal case $\ell(\varphi)$
equals the coefficient $c$ in the expansion $\varphi(z)=z-c/z+\ldots$
near~$\infty$. (Here we pass from $\UD$ to~$\UH$ by means of the Cayley map.)
Similarly, for a semigroup formed by univalent self-maps of the horizontal
strip~$\Pi:=\{z:\Im z\in(0,\pi)\}$ considered in~\cite{Dubovikov},
$c^+(\varphi_{s,t})-c^{-}(\varphi_{s,t})=:\ell(\varphi_{s,t})=t-s$ whenever
$0\le s\le t$, where $c^{\pm}(\varphi):=\lim\big(z-\varphi(z)\big)$ as
$z\to\pm\infty$ within a closed substrip of~$\Pi$. In contrast, for the
semigroup in $\Hol(\Pi,\Pi)$ considered by Goryainov~\cite{Goryainov}, one has
$0\le c^{\pm}(\varphi_{s,t})\le t-s$ and only for the evolution families formed
by the integrals of~\eqref{EQ_non-aut}.

The union $\tilde {\mathcal U}$ of all families~$(\varphi_{s,t})=(z\mapsto
w_{z,s}(t))$ generated by equation~\eqref{EQ_non-aut} forms a semigroup.
However, {\it a priori} $\tilde {\mathcal U}$ does not have to be a dense
subset of~$\mathcal U$. In fact, by the uniqueness theorem for the solutions
to~\eqref{EQ_non-aut}, $\tilde{\mathcal U}$ must contain only univalent
mappings. Hence the closure of $\tilde{\mathcal U}$ in $\Hol(\UD,\UD)$ will be
a proper subset of~$\mathcal U$ if the latter contains non-univalent functions
different from constants. In other cases we can have a kind of opposite
situation when $\tilde{\mathcal U}\supsetneq\mathcal U$. This happen, e.g., for
the semigroup of all normalized ``multislit" mappings, i.e. functions
$\varphi\in\Hol(\UD,\UD)$ with $\varphi(0)=0$ and $\varphi'(0)>0$ such that
$\UD\setminus\varphi(\UD)$ consists of a finite number of Jordan arcs.

However, for the semigroups~$\mathcal U$ studied in the above mentioned works,
the closure of~$\tilde{\mathcal U}$ (or even this set itself) coincides
with~$\mathcal U$, in which case one may say that the semigroup~$\mathcal U$
{\it admits infinitesimal representation} via equation~\eqref{EQ_non-aut}.
Moreover, for {\it a part of}\, these semigroups, it has been proved that
equation~\eqref{EQ_non-aut} establishes a one-to-one correspondence between the
evolution families~$(\varphi_{s,t})$ and the vector fields~$(G_t)$.

\subsection{General approach in Loewner Theory}\label{SS_gen}

Although the chordal and radial versions of the Loewner Theory share common
ideas and structure, on their own they can only be regarded as parallel but
independent theories. The approach of~\cite{Goryainov-Ba, Goryainov,
Goryainov1996, Dubovikov, Goryainov-Kudryavtseva, GoryainovTalk} does show that
there can be (and actually are) much more independent variants of Loewner
evolution bearing similar structures, but does not solve the problem of
constructing a unified theory covering all the cases.

Recently a new unifying approach has been suggested by Bracci and the two first
authors~\cite{BCM1, BCM2}, and further developed in~\cite{AbstractLoew, SMP,
SMP2}. Relying partially on the theory of one-parametric semigroups, which can
be regarded as the autonomous version of Loewner Theory, they have defined
evolution families in {\it the whole semigroup} $\Hol(\UD,\UD)$ as follows.
\begin{definition}(\cite{BCM1})~\label{D_BCM-EF}
A family $(\varphi_{s,t})_{t\ge s\ge0}$ of holomorphic self-maps of the unit
disc is an {\it evolution family of order $d$} with $d\in [1,+\infty]$ (in
short, an {\sl $L^d$-evolution family}) if
\begin{enumerate}
\item[EF1.] $\varphi_{s,s}=\id_{\mathbb{D}}$ for all $s\ge0$,

\item[EF2.] $\varphi_{s,t}=\varphi_{u,t}\circ\varphi_{s,u}$ whenever $0\leq
s\leq u\leq t<+\infty,$

\item[EF3.] for every $z\in\mathbb{D}$ and every $T>0$ there exists a
non-negative function $k_{z,T}\in L^{d}([0,T],\mathbb{R})$ such that
\[
|\varphi_{s,u}(z)-\varphi_{s,t}(z)|\leq\int_{u}^{t}k_{z,T}(\xi)d\xi
\]
whenever $0\leq s\leq u\leq t\leq T.$
\end{enumerate}
\end{definition}
Condition EF3 is to guarantee that any evolution family can be obtained via
solutions of an ODE which resembles both the radial and chordal
Loewner\,--\,Kufarev equations. The vector fields that drive this generalized
Loewner\,--\,Kufarev ODE are referred to as {\it Herglotz vector fields.}

\begin{definition}(\cite{BCM1})~\label{D_BCM-VF}
Let $d\in [1,+\infty]$. A {\it weak holomorphic vector field of order $d$} in
the unit disc $\mathbb{D}$ is a function
$G:\mathbb{D}\times\lbrack0,+\infty)\rightarrow \mathbb{C}$ with the following
properties:

\begin{enumerate}
\item[WHVF1.] for all $z\in\mathbb{D},$ the function $\lbrack
0,+\infty)\ni t\mapsto G(z,t)$ is measurable,

\item[WHVF2.] for all $t\in\lbrack0,+\infty),$ the function $
\mathbb{D}\ni z\mapsto G(z,t)$ is holomorphic,

\item[WHVF3.] for any compact set $K\subset\mathbb{D}$ and  any $T>0$ there
exists a non-negative function $k_{K,T}\in L^{d}([0,T],\mathbb{R})$ such that
\[
|G(z,t)|\leq k_{K,T}(t)
\]
for all $z\in K$ and for almost every $t\in\lbrack0,T].$
\end{enumerate}
We say that $G$ is a {\it (generalized) Herglotz vector field} of order $d$ if,
in addition to conditions~WHVF1\,--\,WHVF3 above, for almost every $t\in
[0,+\infty)$ the holomorphic function $G(\cdot, t)$ is an infinitesimal
generator of a one-parametric semigroup in $\Hol(\UD,\UD)$.
\end{definition}
\begin{remark} By~\cite[Theorem~4.8]{BCM1}, the class of all Herglotz vector
fields of order~$d$ coincides with that of all functions
$G:\UD\times[0,+\infty)\to\Complex$ which can represented in the form
\begin{equation}\label{EQ_BCM-BP}
G(z,t):=(\tau(t)-z)(1-\overline{\tau(t)} z)p(z,t)\quad \text{for all $z\in\D$
and a.e. $t\ge0$},
\end{equation}
where $\tau$ is any measurable function from~$[0,+\infty)$ to $\overline\UD$ and
$p:\UD\times[0,+\infty)\to\Complex$ satisfies the following conditions:
(1)~$p(\cdot,t)$ is holomorphic in~$\UD$ for every~$t\ge0$ with $\Re
p(z,t)\ge0$ for all $z\in\UD$ and $t\ge0$; (2)~$p(z,\cdot)$ is in $L^d_{\rm
loc}([0,+\infty),\Complex)$ for every~$z\in\UD$.

The generalized Loewner\,--\,Kufarev equation
\begin{equation}\label{EQ_BCM-ODE}
\frac{dw}{dt}=G(w,t),~~ t\ge s,\quad w(s)=z,
\end{equation}
resembles the radial Loewner\,--\,Kufarev ODE when $\tau\equiv0$ and
$p(0,t)\equiv1$. Furthermore, with the help of the Cayley map between~$\UD$ and
$\UH$, the chordal Loewner equation appears to be the special case
of~\eqref{EQ_BCM-ODE} with~$\tau\equiv1$.

In a similar way, different notions of evolution families considered previously
in the literature can be reduced to special cases of $L^d$-evolution families
defined above.
\end{remark}

Equation~\eqref{EQ_BCM-ODE} establishes a 1-to-1 correspondence between
evolution families of order~$d$ and Herglotz vector fields of the same order.
Namely, the following theorem takes place.

\begin{result}{\rm (\cite[Theorem~1.1]{BCM1})}\label{TH_BCM-EF-VF} For any evolution family
$(\varphi_{s,t})$ of order $d\in[1,+\infty]$ there exists an (essentially)
unique Herglotz vector field $G(z,t)$ of order $d$ such that for every~$z\in
\D$ and every~$s\ge0$ the function $[s,+\infty)\ni t\mapsto
w_{z,s}(t):=\varphi_{s,t}(z)$ solves the initial value
problem~\eqref{EQ_BCM-ODE}.

Conversely, given any Herglotz vector field $G(z,t)$ of order
$d\in[1,+\infty]$, for every~$z\in \D$ and every~$s\ge0$ there exists a unique
solution $[s,+\infty)\ni t\mapsto w_{z,s}(t)$ to the initial value
problem~\eqref{EQ_BCM-ODE}. The formula $\varphi_{s,t}(z):=w_{z,s}(t)$ for all
$s\ge0$, all $t\ge s$, and all $z\in\UD$, defines an evolution
family~$(\varphi_{s,t})$ of order~$d$.
\end{result}
Here by {\it essential uniqueness} we mean that two Herglotz vector fields
$G_1(z,t)$ and $G_2(z,t)$ corresponding to the same evolution family must
coincide for a.e. $t\ge0$.

The general notion of a Loewner chain has been given\footnote{See
also~\cite{AbstractLoew} for a straightforward extension of this notion to
complex manifolds. It is worth mentioning that, in that paper, the construction
of an associated Loewner chain for a given evolution family is based on
arguments borrowed from Category Theory and so it differs notably from the one
we used in~\cite[Theorems 1.3 and~1.6]{SMP}.} in~\cite{SMP}.
\begin{definition}(\cite{SMP})~\label{D_SMP-LC}
A family $(f_t)_{t\ge0}$ of holomorphic maps of~$\UD$ is called a {\it Loewner
chain of order $d$} with $d\in [1,+\infty]$ (in short, an {\it $L^d$-Loewner
chain}) if it satisfies the following conditions:
\begin{enumerate}
\item[LC1.] each function $f_t:\D\to\C$ is univalent,

\item[LC2.] $f_s(\D)\subset f_t(\D)$ whenever $0\leq s < t<+\infty,$

\item[LC3.] for any compact set $K\subset\mathbb{D}$ and any $T>0$ there exists a
non-negative function $k_{K,T}\in L^{d}([0,T],\mathbb{R})$ such that
\[
|f_s(z)-f_t(z)|\leq\int_{s}^{t}k_{K,T}(\xi)d\xi
\]
whenever $z\in K$ and $0\leq s\leq t\leq T$.
\end{enumerate}
\end{definition}

This definition of (generalized) Loewner chains matches the abstract notion of
evolution family introduced in~\cite{BCM1}. In particular the following
statement holds.
\begin{result}{\rm (\cite[Theorem~1.3]{SMP})}\label{TH_SMP-EF-LC}
For any Loewner chain $(f_t)$ of order $d\in[1,+\infty]$, if we define
$$
\varphi_{s,t}:= f_t^{-1}\circ f_s \quad \text{whenever }0\le s\le t,
$$
then $(\varphi_{s,t})$ is an evolution family of the same order~$d$.
Conversely, for any evolution family $(\varphi_{s,t})$ of
order~$d\in[1,+\infty]$, there exists a Loewner chain $(f_t)$ of the same
order~$d$ such that
\begin{equation*}
 f_t\circ\varphi_{s,t}=f_s\quad  \text{whenever }0\le s\le t.
\end{equation*}
\end{result}
In the situation of this theorem we say that the Loewner chain~$(f_t)$ and the
evolution family~$(\varphi_{s,t})$ are {\it associated with} each other. It was
proved in~\cite{SMP} that given an evolution family~$(\varphi_{s,t})$, an
associated Loewner chain~$(f_t)$ is unique up to conformal maps
of~$\cup_{t\ge0}f_t(\UD)$.

Thus in the framework of the approach described above the essence of Loewner
Theory is represented by the essentially 1-to-1 correspondence among Loewner
chains, evolution families, and Herglotz vector fields.

\begin{remark}\label{RM_EF-uni}
Definition~\ref{D_BCM-EF} does not require elements of an evolution family to
be univalent. However, this condition is satisfied. Indeed, by
Theorem~\ref{TH_BCM-EF-VF}, any evolution family $(\varphi_{s,t})$ can be
obtained via solutions to the generalized Loewner\,--\,Kufarev ODE. Hence the
univalence of $\varphi_{s,t}$'s follows from the uniqueness of solutions to
this~ODE. For an essentially different direct proof see
\cite[Proposition~3]{BCM2}.
\end{remark}

To conclude the section, we mention that an analogous approach has been
suggested for the Loewner Theory in the
annulus~\cite{SMP_annulusI,SMP_annulusII}. However, in this paper we restrict
ourselves to the Loewner Theory for simply connected domains.

\subsection{Aim of the paper and main results}\label{SS_results}

As one can see from stated in Subsect.\,\ref{SS_1}, there are essentially two
different ways to deal with Loewner evolution, regardless of which specific
variant of the theory we consider. One of them involves conformal maps $f_t$
onto an increasing family of simply connected domains described by (one or
another version of) the Loewner equations. The Loewner ODE being equipped with
an initial condition at the left end-point generate, in this case, a
non-autonomous semi-flow consisting of holomorphic self-maps $\varphi_{s,t}$ of
the reference domain ($\UD$ or $\UH$). The other way to study Loewner
evolution, which can be colloquially called ``decreasing", is formally obtained
by reversing the direction of the time parameter~$t$ and changing
correspondingly the sign in the Loewner equations. The initial condition for
the Loewner ODE is again given {\it at the left end-point} and hence in
contrast to the ``increasing" case, the solutions do not extend unrestrictedly
in time. The functions generated by such initial value problem maps a
decreasing family of simply connected domains conformally onto the reference
domain.

It is clear that locally in time there is no principal difference between these
two ways of reasoning: the connection between them is easily established by
considering, for an arbitrary $T>0$, the parameter change $t\mapsto T-t$. At
the same time the diversity of settings within which the Loewner evolution has
been considered in literature may cause certain difficulties. The primary aim
of this paper is to describe rigorously how the general approach discussed in
the previous subsection can be ``translated" to the ``decreasing" setting.
First of all we introduce ``decreasing" analogues of Loewner chains and
evolution families.

For a set $E\subset \Real$ we denote by $\Delta(E)$ the ``upper triangle"
$\{(s,t):s,t\in E,\,s\le t\}$.

\begin{definition}\label{D_decLC} Let $d\in [1,+\infty]$. A family
$(f_t)_{t\ge0}$ of holomorphic self-maps of the unit disk~$\UD$ will be called
a {\sl decreasing Loewner chain of order $d$} (or, in short, {\it decreasing
$L^d$-chain})  if it satisfies the following conditions:
\begin{enumerate}
\item[DC1.] each function $f_t:\D\to\D$ is univalent,

\item[DC2.] $f_0=\id_\UD$ and $f_t(\D)\subset f_s(\D)$ for all $(s,t)\in\Delta\big([0,+\infty)\big)$,

\item[DC3.] for any compact set $K\subset\mathbb{D}$ and  all $T>0$ there exists a
non-negative function $k_{K,T}\in L^{d}\big([0,T],\mathbb{R}\big)$ such that
\[
|f_s(z)-f_t(z)|\leq\int_{s}^{t}k_{K,T}(\xi)d\xi
\]
for all $z\in K$ and all $(s,t)\in \Delta\big([0,T]\big)$.
\end{enumerate}
\end{definition}
\begin{remark}
Note that the difference of this definition from that of an increasing Loewner
chain resides not only in the opposite inclusion sign in condition DC2. We also
assume that~$f_0=\id_\UD$, which we always can do by composing with a conformal
mapping, while in the ``increasing" case no such restriction is imposed, because
in general, it may be impossible to employ any conformal mapping except for
linear transformations of~$\Complex$.
\end{remark}
\begin{remark}
As we will see in Sect.~\ref{S_decChains}, by means of the Cayley map between
$\UD$ and $\UH$, the notion of {\it chordal Loewner families} introduced by
Bauer~\cite{Bauer} reduces to a special case of decreasing Loewner chains of
order~$d=+\infty$.
\end{remark}

Now we introduce reverse evolution families, the ``decreasing" counterparts of
$L^d$-evolution families.
\begin{definition}\label{D_Rev_EF} Let $d\in [1,+\infty]$. A family
$(\varphi_{s,t})_{0\leq s\leq t}$ of holomorphic self-maps of the unit
disk~$\UD$ is a {\it reverse evolution family of order $d$} if the following
conditions are fulfilled:
\begin{enumerate}
\item[REF1.] $\varphi_{s,s}=\id_{\mathbb{D}},$

\item[REF2.] $\varphi_{s,t}=\varphi_{s,u}\circ\varphi_{u,t}$ for all $0\leq
s\leq u\leq t<+\infty,$

\item[REF3.] for any  $z\in\mathbb{D}$ and any $T>0$ there exists a
non-negative function $k_{z,T}\in L^{d}([0,T],\mathbb{R})$ such that
\[
|\varphi_{s,u}(z)-\varphi_{s,t}(z)|\leq\int_{u}^{t}k_{z,T}(\xi)d\xi
\]
for all $s,t,u\in[0,T]$ satisfying inequality $s\le u\le t$.
\end{enumerate}
\end{definition}
\begin{remark}
Speaking informally, comparing with Definition~\ref{D_BCM-EF}, the parameters
$s$ and $t$ switch their roles. That is why condition EF3 is not converted to
condition REF3 under the ``time reversing trick". This difficulty does not
appear however in the classical settings such as radial and chordal Loewner
evolutions.
\end{remark}

Using the results of~\cite{BCM1, SMP} we deduce relations among decreasing
Loewner chains, reverse evolution families and Herglotz vector fields. In
particular, we obtain the following extension of Theorem~\ref{TH_LawlerBook}
(and its more abstract form in~\cite[Chapter~4]{LawlerBook}).
\begin{theorem}\label{TH_G-f_t}
Let $G$ be a Herglotz vector field of order $d\in[1,+\infty]$. Then:
\begin{mylist}
\item[(i)]
For every $z\in\UD$, there exists a unique maximal\,\footnote{In this case
``maximal", or ``non-extendable", means that there are no solutions $[0,T)\ni
t\mapsto \tilde w_z(t)\in\UD$ to~\eqref{ode1}, with $T\in(0,+\infty]$, such
that $w_z$ is the restriction of $\tilde w_z$ to a proper subset of~$[0,T)$.}
solution~$w(t)=w_z(t)\in\UD$ to the following initial value problem
\begin{equation}\label{ode1}
\frac{d w}{dt}=-G(w,t), \qquad w(0)=z\;.
\end{equation}
\item[(ii)] For every  $t\ge0$, the set $\Omega_t$ of all $z\in\UD$ for which
$w_z$ is defined at the point\,~$t$ is a simply connected domain and the
function $g_t:\Omega_t\to\UD$ defined by $g_t(z):=w_z(t)$ for all
$z\in\Omega_t$, maps $\Omega_t$ conformally onto $\UD$.

\item[(iii)] The functions $f_t:=g_t^{-1}:\UD\to\Omega_t$ form a decreasing Loewner chain
of order~$d$, which is the unique solution to the following initial value
problem for the Loewner\,--\,Kufarev PDE:
\begin{equation}\label{EQ_LK-PDE-ini}
\frac{\partial f_t(z)}{\partial t}=\frac{\partial f_t(z)}{\partial z}
G(z,t),\qquad f_0=\id_\UD.
\end{equation}
\end{mylist}
\end{theorem}
\begin{remark}
Equation~\eqref{ode1}, as well as equation~\eqref{EQ_BCM-ODE} in
Subsect.\,\ref{SS_gen}, has discontinuous in $t$ right-hand side and hence
should be understood as a Carath\'eodory ODE, see e.g. \cite[\S I.1]{Filippov},
\cite[Ch.\,18]{Kurzweil} for the general theory of such equations,
or~\cite[\S2]{SMP_annulusI} for the basic results in the case of the right-hand
side holomorphic w.r.t.~$w$. As for PDE~\eqref{EQ_LK-PDE-ini} there does not
seem to be any abstract theory that suits completely well our purposes. We
discuss the notion of a solution to~\eqref{EQ_LK-PDE-ini} in
Subsect.\,\ref{SS_solution}.
\end{remark}

Some of the statements we get are obtained by careful "translating" of already
known results under the time parameter change, while other statements, in
particular those involving time regularity of evolution families, require
deeper ideas. As an interesting "byproduct" we prove the following assertion.

Denote by $AC^d(X,Y)$, where $X\subset \Real$ and $Y\subset\Complex$, the class
of all locally absolutely continuous functions $f:X\to Y$ whose derivative is
of class $L^d_{\rm loc}$.
\begin{theorem}\label{TH_LC}
Let $(f_t)_{t\ge0}$ be a family of holomorphic functions in the unit disc~$\UD$ satisfying
conditions LC1 and LC2. Then $(f_t)$ is an $L^d$-Loewner chain if and only if
\begin{itemize}
\item[LC3w.] For any  $T>0$ there exist two distinct points $\zeta_1,\zeta_2\in\UD$
such that the mappings $[0,T]\ni t\mapsto f_t(\zeta_j)$, $j=1,2$, both belong
to~$AC^d([0,T],\UD)$.
\end{itemize}
\end{theorem}
\begin{remark}
It follows immediately from Definitions~\ref{D_SMP-LC} and~\ref{D_decLC} that
the above theorem is equivalent to the following statement: {\it a family
$(f_t)_{t\ge0}$ of holomorphic functions $f_t:\UD\to\Complex$ satisfying
conditions DC1 and DC2 is a decreasing Loewner chain of order~$d$ if and only
if the above condition LC3w holds.}
\end{remark}

The rest of the paper is organized as follows. In Section~\ref{S_aux} we make
precise definition of what is meant by solutions to the generalized
Loewner\,--\,Kufarev PDE~\eqref{EQ_LK-PDE-ini} and prove some auxiliary
propositions and lemmas.

In Section~\ref{S_decChains} we discuss relationship between decreasing Loewner
chains and Herglotz vector fields. In particular, we prove
Theorem~\ref{TH_G-f_t}. We also establish a kind of inverse theorem
(Theorem~\ref{TH_f_t-G}).

Section~\ref{S_REF} is devoted to the study of the relationship of reversed
evolution families on the one side, and decreasing Loewner chains together with
the corresponding Herglotz vector fields on the other side.

Finally in Section~\ref{S2} we prove Theorem~\ref{TH_LC}.

\section{Auxiliary statements}\label{S_aux}
\subsection{Solutions to the generalized Loewner\,--\,Kufarev PDE}\label{SS_solution}
Consider the (generalized) Loewner\,--\,Kufarev PDE
\begin{equation}\label{EQ_LK-PDE}
\frac{\partial F(z,t)}{\partial t}=\epsilon\,\frac{\partial F(z,t)}{\partial z}
G(z,t),
\end{equation}
where $G$ is a Herglotz vector field and $\epsilon$ is a constant in
$\{-1,+1\}$ whose value depends on whether we deal with the increasing
($\epsilon\equiv-1$) or decreasing ($\epsilon\equiv1$) variant of the theory.

The formulation of Theorems~\ref{TH_G-f_t}, \ref{TH_f_t-G},
and~\ref{TH_REF-EQS} contains the notion of a solution to the
Loewner\,--\,Kufarev PDE, which should be defined more precisely. Similar to
the classical text~\cite[Chapter 6]{Pommerenke}, we give the definition as
follows.
\begin{definition}\label{D_solution}
By a solution to the Loewner\,--\,Kufarev PDE equation~\eqref{EQ_LK-PDE} we
mean a function $F:\UD\times E\to\Complex$, where $E\subset [0,+\infty)$ is an
interval, such that
\begin{mylist}
\item[S1.] $F$ is continuous in~$\UD\times E$;
\item[S2.] for every $t\in E$ the function $F(\cdot,t)$ is holomorphic in~$\UD$;
\item[S3.] for every $z\in\UD$ the function $F(z,\cdot)$ is locally absolutely continuous in~$E$;
\item[S4.] for every $z\in\UD$  equality~\eqref{EQ_LK-PDE} holds a.e. in $E$.
\end{mylist}
\end{definition}
\begin{remark}
Condition S4 means exactly the following: for each $z\in\UD$ there exists a
null-set $M_z\subset E$ such that~\eqref{EQ_LK-PDE} holds for all $z\in\UD$ and
all $t\in E\setminus M_z$. Note that {\it a priori} the sets $M_z$ can depend
on~$z$. We show below that under conditions S1\,--\,S3 this set can be chosen
{\it independently} of $z$.
\end{remark}

We complete this subsection with the proof of the following technical
statement, which might appear implicitly, in one or another context, in some
works on Loewner Theory.
\begin{proposition}\label{PR_characteristics}
Let $\epsilon\in\{-1,+1\}$ and let $G$ be a Herglotz vector field of some
order~$d\in[1,+\infty]$. Suppose that a function $F:\UD\times E\to\Complex$
satisfies conditions S1\,--\,S4. Then the following assertions hold:
\begin{mylist}
\item[(i)] The function $F(z,t)$ is locally absolutely continuous in\,~$t$
uniformly w.r.t. $z$ on every compact subset of~$\UD$, i.e. for any compact set
$K\subset \UD$ and any compact interval $I\subset E$ the mapping $t\mapsto
F(\cdot,t)\in\Hol(\UD,\C)$ is absolutely continuous on~$I$ w.r.t. the metric
$d_K(f,g):=\max_{z\in K}|f(z)-g(z)|$.

\item[(ii)] There is a null-set $N\subset [0,+\infty)$
such that for all $z\in\UD$ and all $t\in E\setminus N$ the partial derivative
$\partial F(z,t)/\partial t$ exists and equality~\eqref{EQ_LK-PDE} holds.
Moreover, for each $t\in E\setminus N$, $\big(F(z,t+h)-F(z,t)\big)/h\to\partial
F(z,t)/\partial t$ locally uniformly in $\UD$ as $h\to0$.

\item[(iii)] For any solution $t\mapsto w(t)\in\UD$ to the generalized Loewner\,--\,Kufarev ODE
\begin{equation}\label{EQ_LK_ODE_eps}
\dot w=-\epsilon\, G(w,t)
\end{equation}
the function $t\mapsto F\big(w(t),t\big)$ is constant in the domain of\, $w$
intersected with $E$.
\end{mylist}
\end{proposition}
\begin{proof}
By S3 for any $z\in\UD$ the map $F(z,\cdot)$ is locally absolutely continuous
on~$E$. Hence by S4, for any $z\in \mathbb D$ and any $(s,t)\in\Delta(E)$,
\begin{equation}\label{EQ_viaDer}
F(z,t)-F(z,s)=\int_s^t \frac{\partial F(z,\xi)}{\partial
t}\,d\xi=\epsilon\int_s^t\frac{\partial F(z,\xi)}{\partial z} G(z,\xi)\,d\xi.
\end{equation}

Fix now any compact interval $I\subset E$ and any closed disk $K:=\{z:|z|\le
r\}$, with some $r\in(0,1)$. From S1 it follows that
$\big(F(\cdot,t)\big)_{t\in I}$ forms a compact subset in~$\Hol(\UD,\Complex)$.
Therefore there exists a constant $C>0$ such that $|\partial F(z,t)/\partial
z|\le C$ for all $z\in K$ and all $t\in I$.  Combined with
inequality~\eqref{EQ_viaDer} and condition WHVF3 from Definition~\ref{D_BCM-VF}
of a Herglotz vector field, this fact implies that there exists a non-negative
function $k_{I,K}\in L^d\big(I,\Real\big)$ such that
\begin{equation}\label{EQ_forLCh}
|F(z,t)-F(z,s)|\le C\int_s^t|G(z,\xi)|\,d\xi\le C\int_s^tk_{I,K}(\xi)\,d\xi
\end{equation}
for any $z\in K$ and any $(s,t)\in\Delta(I)$. Assertion~(i) follows now
immediately.

Applying the above argument to a sequence of closed disks $(K_n\subset\UD)$ and
a sequence of compact intervals $(I_k\subset E)$ whose unions cover $\UD$ and
$E$, respectively, one can easily construct a sequence of non-negative
functions $k_n\in L^d_{\rm loc}(E,\Real)$ such that for each $n\in\Natural$,
\begin{equation}\label{EQ_k_n}
|F(z,t)-F(z,s)|\le  \int_s^tk_n(\xi)\,d\xi
\end{equation}
for all  $(s,t)\in\Delta(E)$ and all $z\in K_n$. Choose any $t_0\in E$. Since
the functions
$$
Q_n(t):=\int_{t_0}^tk_n(\xi)\,d\xi, \quad t\in E,~n\in\Natural,
$$
are absolutely continuous, there exists a null-set $N_0\subset E$ such that
$Q_n'(t)$ exists finitely for all $t\in E\setminus N_0$ and all $n\in\Natural$.
Now from~\eqref{EQ_k_n} it follows that for each $t\in E\setminus N_0$ and some
$\varepsilon>0$ small enough the family $$\mathcal
F_t:=\left\{\frac{F(\cdot,t')-F(\cdot,t)}{t'-t}:t'\in
E,\,0<|t'-t|<\varepsilon\right\}$$ is bounded on each of $K_n$'s. Hence
$\mathcal F_t$ is relatively compact in $\Hol(\UD,\Complex)$ provided ${t\in
E\setminus N_0}$. For each $z\in\UD$ by $M_z$ we denote the null set in~$E$
aside which $F$ satisfies~\eqref{EQ_LK-PDE}. Following a standard technique we
apply now Vitali's principle to the family $\mathcal F_t$ and the set
$\{z_j:=1-1/(j+1):j\in\Natural\}$ in order to conclude that if $t\in E\setminus
N_0$ and $t\not\in\cup_{j\in\Natural}M_{z_j}$, then
$$\frac{F(z,t')-F(z,t)}{t'-t}\to\epsilon\,\frac{\partial F(z,t)}{\partial
z}G(z,t)$$ locally uniformly in~$\UD$ as $t'\to t$, $t'\in E$. This proves~(ii)
with $N:=N_0\cup \big(\cup_{j\in\Natural}M_{z_j}\big)$.

To prove (iii), we note that due to compactness of $\{F(\cdot,t)\}_{t\in I}$
for any compact interval~$I\subset E$, the limit
\begin{equation}\label{EQ_limit}
\lim_{\zeta\to z}\frac{F(\zeta,t)-F(z,t)}{\zeta-z}=\frac{\partial
F(z,t)}{\partial z}
\end{equation}
is attained uniformly w.r.t. $t\in I$ for any fixed $z\in\UD$. This justifies
the formal computation
\begin{multline*}
\frac{d F(w(t),t)}{dt}=\dot w(t)\left.\frac{\partial F(z,t)}{\partial
z}\right|_{z:=w(t)}+\left.\frac{\partial F(z,t)}{\partial
t}\right|_{z:=w(t)}\\=-\epsilon\, G(w(t),t)\left.\frac{\partial
F(z,t)}{\partial z}\right|_{z:=w(t)}+\left.\left(\epsilon\,\frac{\partial
F(z,t)}{\partial z}G(z,t)\right)\right|_{z:=w(t)}=0
\end{multline*}
for any $z\in\UD$ and a.e. $t\in E$. From (i) and compactness of
$\{F(\cdot,t)\}_{t\in I}$ for compact intervals~$I\subset E$ it follows that
$t\mapsto F(w(t),t)$ is locally absolutely continuous in its domain. Thus we
may conclude that this function is constant. The proof is now complete.
\end{proof}

\subsection{Some lemmas}
In what follows will take advantage of several lemmas proved in this
subsection. We start with a kind of ``rigidity" lemma, going back to the Schwarz
Lemma and the classical growth estimate for holomorphic functions with positive
real part.

Denote by $\rho_\UD$ the pseudohyperbolic distance in~$\UD$, i.e. let
$$
\rho_\UD(z,w):=\left|\frac{w-z}{1-\bar w\,z}\right|,\quad\text{for all
$z,w\in\UD$}.
$$
\begin{lemma}\label{LM_pseudo}
There exists a universal constant $C>0$ such that for any
$\psi\in\Hol(\UD,\UD)$ with $\psi(0)=0$, any $r\in(0,1)$, and any
$\zeta_0\in\UD\setminus\{0\}$,
\begin{equation}\label{EQ_psi}
\rho_\UD\big(\psi(\zeta),\zeta\big)\le\frac{C|\psi(\zeta_0)-\zeta_0|}{|\zeta_0|(1-|\zeta_0|^2)(1-r)^2}
\end{equation}
whenever $|\zeta|\le r$.
\end{lemma}
\begin{proof}
The proof is based on~\cite[Lemma 3.8]{SMP_annulusI}, according to which
$$
\big|\psi(\zeta)-\zeta\big|\le\frac{C|\psi(\zeta_0)-\zeta_0|}{|\zeta_0|(1-|\zeta_0|^2)(1-r^2)}
\quad\text{whenever $|\zeta|\le r$,}
$$
where $C>0$ is a universal constant.

From the Schwarz Lemma it follows that $\big|1-\bar \zeta\,\psi(\zeta)\big|\ge
1-|\zeta|^2$. Hence to deduce~\eqref{EQ_psi} it remains to notice that
$1/(1-r^2)\le1/(1-r)$.
\end{proof}

\begin{corollary}\label{C_pseudo}
For any $\varphi\in\Hol(\UD,\UD)$, any $\zeta, \zeta_0\in\UD\setminus\{0\}$, and any
$r\in(0,1)$,
\begin{equation}\label{EQ_pseudo}
\rho_\UD\big(\varphi(\zeta),\zeta\big)\le
C\frac{|\varphi(\zeta_0)-\zeta_0|+4|\varphi(0)|}{|\zeta_0|(1-|\zeta_0|)^2(1-r)^2}
\quad\text{whenever $|\zeta|\le r$,}
\end{equation}
where $C$ is the universal constant from Lemma~\ref{LM_pseudo}.
\end{corollary}
\begin{proof}
Apply Lemma~\ref{LM_pseudo} for the function $\psi:=\ell\circ\varphi$, where
$$\ell(z):=\frac{z-w_0}{1-\overline{w}_0z},\quad w_0:=\varphi(0).$$
Let us estimate first $|\ell(z)-z|$. For any $z\in \UD$ we have
\begin{equation}\label{EQ_ell(z)-z}
\big|\ell(z)-z\big|=\left|\frac{\overline w_0 z^2 - w_0}{1-\overline w_0
z}\right|\le\frac{2|w_0|}{1-|z|}=\frac{2|\varphi(0)|}{1-|z|}.
\end{equation}
Similarly, $|w-\ell^{-1}(w)|\le 2|\varphi(0)|/(1-|w|)$ for all $w\in\UD$.
Setting $w:=\psi(\zeta_0)$ and bearing in mind that in this case
$|w|\le|\zeta_0|$ by the Schwarz Lemma, we therefore conclude that
\begin{equation}\label{EQ_psi(z_0)-z_0}
\big|\psi(\zeta_0)-\zeta_0\big|\le
\big|\varphi(\zeta_0)-\zeta_0\big|+\big|\psi(\zeta_0)-\ell^{-1}\big(\psi(\zeta_0)\big)\big|\le
\big|\varphi(\zeta_0)-\zeta_0\big|+\frac{2|\varphi(0)|}{1-|\zeta_0|}.
\end{equation}

By the invariance of the pseudohyperbolic distance under the M\"obius
transformations of~$\UD$ and by the triangle inequality, for each $\zeta\in\UD$
we have
$$\rho_\UD\big(\varphi(\zeta),\zeta\big)=\rho_\UD\big(\psi(\zeta),\ell(\zeta)\big)
=\rho_\UD(\ell(\zeta),\zeta)+\rho_\UD\big(\psi(\zeta),\zeta\big).$$
Inequality~\eqref{EQ_ell(z)-z} implies that
$\rho_\UD\big(\ell(\zeta),\zeta\big)\le\big|\ell(\zeta)-\zeta\big|/\big(1-|\zeta|\big)\le
2|\varphi(0)|/\big(1-|\zeta|\big)^2$, while the estimate for
$\rho_\UD\big(\psi(\zeta),\zeta\big)$ is obtained from~\eqref{EQ_psi(z_0)-z_0}
and~\eqref{EQ_psi}.

Now~\eqref{EQ_pseudo} follows easily.
\end{proof}
\begin{corollary}\label{C_pseudo_symm} For every $r,R,\rho\in(0,1)$ there exists a constant
$\tilde C=\tilde C(r,R,\rho)$ such that for any $\varphi\in\Hol(\UD,\UD)$,
\begin{equation}\label{EQ_pseudo_symm}
\rho_{\UD}\big(\varphi(\zeta),\zeta\big)\le \tilde
C(r,R,\rho)\big(|\varphi(\zeta_1)-\zeta_1|+|\varphi(\zeta_2)-\zeta_2|\big)
\end{equation}
whenever $\zeta,\zeta_1,\zeta_2\in\UD$ satisfy the conditions $|\zeta|\le r$,
$|\zeta_j|\le R$, $j=1,2$, and $\rho_\UD(\zeta_2,\zeta_1)\ge \rho$.
\end{corollary}
\begin{proof}
The statement of the corollary can be obtained in the following way. Applying
Corollary~\ref{C_pseudo} for $\ell^{-1}\circ\varphi\circ\ell$ and
$\zeta_0:=\ell^{-1}(\zeta_2)$, where $\ell$ is a M\"obius automorphism
of~$\UD$ sending $0$ to~$\zeta_1$, one obtains an estimate for
$\rho_\UD\big(\varphi(\ell(\zeta)),\ell(\zeta)\big)=
\rho_\UD\big((\ell^{-1}\circ\varphi\circ\ell)(\zeta),\zeta\big)$.

Now substitute $\ell^{-1}(\zeta)$ for $\zeta$ in order to
deduce~\eqref{EQ_pseudo_symm}. To carry out these estimates, we also use the
fact that $|\ell^{-1}(z)|\le(|z|+|\zeta_1|)/(1+|z\zeta_1|)$ and
$|(\ell^{-1})'(z)|\le 2/(1-|\zeta_1|)$ for all $z\in\UD$. Since the concrete
expression for the constant~$\tilde C(r,R,\rho)$ is not important for our
purposes, we omit the details.
\end{proof}

The following lemma seems to contain well known facts. We include it only for
the sake of completeness. In what follows we will write $B\csub A$ to indicate
that the closure of $B$ is a compact subset of $A$.
\begin{lemma}\label{LM_seq-conf}
Let $(f_n)$ be a sequence of univalent holomorphic
functions~$f_n:\UD\to\Complex$ converging locally uniformly in~$\UD$ to a
non-constant function~$f$. Then:
\begin{itemize}
\item[(i)] the function $f$ is univalent in~$\UD$;
\item[(ii)] for each set $U\csub D:=f(\UD)$ there exists
$n_0=n_0(U)$ such that for all $n>n_0(U)$, $n\in\Natural$, we have $U\subset
D_n:=f_n(\UD)$;
\item[(iii)] for each $U\csub D:=f(\UD)$ the sequence
$g_k:=f_{n_0(U)+k}^{-1}$ converges uniformly on~$U$ to $g:=f^{-1}$.
\end{itemize}
\end{lemma}
\begin{proof}
Statement (i) is (a variant of) the Hurwitz Theorem, see, e.\,g.,
\cite[p.\,5]{Duren}. Statement (ii) can be easily deduced from Rouche's
Theorem. We give here only the proof of~(iii).

First of all we can assume that $U$ is a domain. Furthermore, it is sufficient
to prove only locally uniform convergence of $g_k$ to $g$ on $U$.

The functions $g_k$ satisfy $|g_k(z)|<1$ for any $z\in U$. Hence $(g_k)$ is a
normal family in $U$ and we only need to show that the limit of any convergent
subsequence of~$(g_k)$ coincides on $U$ with~$g$. So without loss of generality
we can assume that the sequence $(g_k)$ itself converges locally uniformly
on~$U$ to some holomorphic function $\tilde g:U\to\overline\UD$.

We claim that actually
\begin{equation}\label{EQ_inD}
\tilde g(U)\subset\UD.
\end{equation}
To show that this is the case we fix any $r\in(0,1)$ such that $U\subset\subset
f(r\UD)$. Then by~(ii), $U\subset f_n(r\UD)$ for all $n\in\Natural$ large
enough. Hence $\tilde g(U)\subset\overline{r\UD}\subset \UD$.

According to~\eqref{EQ_inD} we can pass to limits as $k\to+\infty$ in the
equality ${f_{n_0(U)+k}\circ (g_k|_U)=\id_U}$, $k\in\Natural$, to see that
$f\circ\tilde g=\id_U$, i.e., $\tilde g$ coincides on~$U$ with $g$.
\end{proof}

\begin{lemma}\label{LM_phi-nepr}
Let $(f_t)_{t\in[0,T]}$ be a family of holomorphic univalent functions
${f_t:\UD\to\Complex}$, satisfying the following conditions:
\begin{itemize}
\item[(i)] $\{f_t:t\in[0,T]\}$ is a normal family in~$\UD$;
\item[(ii)] there exist two points $\zeta_1,\zeta_2\in\UD$, $\zeta_1\neq \zeta_2$,
such that the functions $[0,T]\ni t\mapsto f_t(\zeta_j)\in\Complex$ are
continuous for $j=1,2$.
\end{itemize}
Then for any compact set $K\subset\UD$ there exists a constant $M_K>0$ such
that
\begin{equation}\label{EQ_phi-nepr}
 |z_1-z_2|\le M_K|f_t(z_1)-f_t(z_2)|
\end{equation}
for any $t\in[0,T]$, any $z_1\in K$ and all $z_2\in\UD$.
\end{lemma}
\begin{proof}
Assume the conclusion is false. Then there exist sequences $(z_n^{(1)})$, $(z_n^{(2)})$,
$(t_n)$ and a compact set $K\subset\UD$ such that for all $n\in\Natural$ we
have:
\begin{itemize}
\item[(a)] $z_n^{(1)}\in K$, $z_n^{(2)}\in \UD$ and $t_n\in[0,T]$;

\item[(b)]
$|z_n^{(1)}-z_n^{(2)}|>n|f_{t_n}(z_n^{(1)})-f_{t_n}(z_n^{(2)})|$.
\end{itemize}

Since $|z_n^{(1)}-z_n^{(2)}|<2$, from (b) it follows that
\begin{itemize}
\item[(c)] $|f_{t_n}(z_n^{(1)})-f_{t_n}(z_n^{(2)})|\to 0$ as $n\to+\infty$.
\end{itemize}

Recall that by (i), $(f_t)$ constitutes a normal family in~$\UD$. Hence using
(a) and passing if necessary to subsequences, we can assume that
\begin{itemize}
\item[(d)] $z^{(1)}_n\to z_0$ and $f_{t_n}\to f$ locally uniformly in~$\UD$ as $n\to+\infty$
\end{itemize}
for some $z_0\in K$ and some $f\in\Hol(\UD,\ComplexE)$.

By condition~(ii) the functions $[0,T]\ni t\mapsto f_t(z_j)$, $j=1,2$, are
continuous. Moreover, $f_t(z_1)\neq f_t(z_2)$ for all $t\in[0,T]$. Therefore,
$|f_{t_n}(z_1)-f_{t_n}(z_2)|>m$ and $|f_{t_n}(z_1)|<M$ for all $n\in\Natural$
and some constants $m>0$ and $M>0$ not depending on~$n$. Hence $f$ is a
holomorphic function in~$\UD$ different from a constant.
By~Lemma~\ref{LM_seq-conf}~(i), $f$ is univalent~in~$\UD$.

Now fix any closed disk~$U$ centered at~$w_0:=f(z_0)$ and lying in~$f(\UD)$ and
let $U'$ be a closed disk of smaller radius centered also at~$w_0$. Then
according to (c), (d) and Lemma~\ref{LM_seq-conf}~(ii) there exists $n_0$ such
that for all $n>n_0$, $n\in\Natural$ we have
\begin{itemize}
\item[(e)] $f_{t_n}(z^{(j)}_n)\in U'$, $j=1,2$, and $U\subset
f_{t_n}(\UD)$.
\end{itemize}
Hence from Lemma~\ref{LM_seq-conf}~(iii) it follows that $(f_{t_n}^{-1})'\to
(f^{-1})'$ uniformly on~$U'$ as ${n\to+\infty}$, $n>n_0$. Therefore there
exists a constant~$M>0$ such that $|(f_{t_n}^{-1})'(w)|<M$ for all $n>n_0$,
$n\in\Natural$ and all $w\in U'$. Finally, taking into account (e), we get $$
|z_n^{(1)}-z_n^{(2)}|<M|f_{t_n}(z_n^{(1)})-f_{t_n}(z_n^{(2)})| $$ for all
$n>n_0$, $n\in\Natural$. This contradicts assumption~(b) and hence completes
the proof of~\eqref{EQ_phi-nepr}.
\end{proof}

\section{Decreasing Loewner chains and Herglotz vector Fields}
\label{S_decChains} As we mentioned in the Introduction, while the classical
Loewner Theory deals with increasing Loewner chains over $[0,+\infty)$, the SLE
theory, having recently caused a burst of interest to Loewner Theory, is based
on the ``decreasing" counterpart of the classical constructions
of~\cite{Loewner, Kufarev, Pommerenke-65, Pommerenke, Kufarev_etal, Aleks1983,
AleksST, Goryainov-Ba}. The variant of the Loewner ODE underlying the chordal
SLE is the following equation
\begin{equation}\label{EQ_classChDec}
\dot w=\frac{2}{w-\lambda(t)},\quad t\ge0,~ w(0)=\zeta,
\end{equation}
where the initial condition $\zeta$ is chosen in the upper half-plane
$\UH:=\{\zeta:\Im \zeta>0\}$ and $\lambda:[0,+\infty)\to\Real$ is a continuous
control function. The stochastic ODE describing SLE is obtained by substituting
the Brownian motion $(B_t)$ times a positive factor for~$\lambda(t)$.

\begin{remark}
A more general (deterministic) form of~\eqref{EQ_classChDec} was studied by
Bauer~\cite{Bauer}.
\end{remark}

By means of the Cayley map $H(z):=i(1+z)/(1-z)$ from $\UD$ onto $\UH$
equation~\eqref{EQ_classChDec} can be rewritten as $\dot w=-G_{\lambda(t)}(w)$,
where
$G_{\lambda}(z):=(1/2)(1-z^3)\big(1-H^{-1}(\lambda)\big)/\big(z-H^{-1}(\lambda)\big)$
for all $z\in\UD$ and all $\lambda\in\Real$. With $\lambda(t)$ being
continuous, it is easy to see that $\UD\times[0,+\infty)\ni(z,t)\mapsto
G_{\lambda(t)}(z)$ is a Herglotz vector field of order~$d=+\infty$ (see
Definition~\ref{D_BCM-VF}). More generally,  although the definitions given
in~\cite{Aleks1983, AleksST, AleksSTSob, Goryainov-Ba, Bauer} differ, the
authors of these works considered essentially the same class $\mathfrak R$ of
non-autonomous vector fields in $\UH$ consisting of functions
$p:\UH\times[0,+\infty)\to\UH$ measurable in $t$ for each $z\in\UH$ and
representable, for a.e. $t\ge0$ fixed, in the following form
$$
p(\zeta,t)=\int_\Real\frac{d\mu_t(x)}{x-\zeta},
$$
where $\mu_t$, for each $t\ge0$, is a probability measure on~$\Real$. Via the
Cayley map, this class corresponds to the vector fields in $\UD$ given by
$G_p(z,t)=p(H(z),t)/H'(z)$. Using the estimate $|p(\zeta,t)|\le 1/\Im \zeta$
for all $\zeta\in\UH$, a.e. $t\ge0$ and all $p\in\mathfrak R$ (see, e.g.,
\cite[Lemma~1]{AleksSTSob}) it is easy to see that the vector fields $G_p$
represent again a particular case of Herglotz vector fields of
order~$d=+\infty$.

Thus Theorem~\ref{TH_G-f_t}, which we are going to prove in this section, can
regarded as an extension of Theorem~\ref{TH_LawlerBook} as well as of its more
abstract form in~\cite[\S4.1]{LawlerBook}.

We also will establish a kind of inverse statement for Theorem~\ref{TH_G-f_t}.
\begin{theorem}\label{TH_f_t-G}
Suppose $(f_t)$ is a decreasing Loewner chain of order~$d\in[1,+\infty]$.
Denote $\Omega_t:=f_t(\UD)$, $g_t:=f^{-1}_t:\Omega_t\to\UD$ for all $t\ge0$,
and $t(z):=\sup\{t\ge0:z\in\Omega_t\}$ for all $z\in\UD$. Then there exists a
Herglotz vector field~$G$ of order~$d$ and a null-set $N\subset[0,+\infty)$
such that the following three statements hold.
\begin{mylist}
\item[(i)] For every $t\in[0,+\infty)\setminus N$ the function
$$
z\mapsto \frac{\partial g_t(z)}{\partial
t}:=\lim_{h\to0}\frac{g_{t+h}(z)-g_t(z)}{h}
$$
is well defined and holomorphic in~$\Omega_t$. Moreover, for every $z\in\UD$
the function $[0,t(z))\ni t\mapsto w_z(t):=g_t(z)$ is the maximal solution to
the following initial value problem for the generalized Loewner\,--\,Kufarev
ODE
\begin{equation}\label{EQ_dec-ini-ODE}
\frac{dw}{dt}=-G(w,t),~~~t\ge0,\quad w(0)=z.
\end{equation}
Substituting $w_z(t)$ for $w(t)$ turns this equation into equality that holds
for all ${t\in[0,t(z))\setminus N}$.

\item[(ii)] The function $F(z,t):=f_t(z)$, $t\ge0$, $z\in\UD$, is a solution\footnote{For
the definition of solutions to this equation, see Subsect.\,\ref{SS_solution}.}
to the following generalized Loewner\,--\,Kufarev PDE
\begin{equation}\label{EQ_dec-PDE}
\frac{\partial F(z,t)}{\partial t}=\frac{\partial F(z,t)}{\partial z} G(z,t)
\end{equation}
with the initial condition $F(\cdot,0)=\id_\UD$. Substituting $f_t(z)$ for
$F(z,t)$ turns this equation into equality that holds for all
${t\in[0,+\infty)\setminus N}$ and all $z\in\UD$.

\item[(iii)] Given any holomorphic function $F_0:\UD\to\Complex$, the initial value problem
$F(\cdot,0)=F_0$ for PDE~\eqref{EQ_dec-PDE} has a unique solution
${(z,t)\mapsto F(z,t)\in\Complex}$, which is defined for all
$(z,t)\in\UD\times[0,+\infty)$ and is given by the formula $F(\cdot,t)=F_0\circ
f_t$, $t\in[0,+\infty)$.
\end{mylist}

The Herglotz vector field $G$ for which at least one of  these statements
holds\footnote{The null-set $N$ depends, of course, on the choice of $G$:
changing $G$ on a null-set will change $N$.} is essentially unique, i.e. any
two such Herglotz vector fields should agree for all $t\in[0,+\infty)\setminus
M$ and all $z\in\UD$, where $M\subset[0,+\infty)$ is a null-set.
\end{theorem}

\begin{remark}
In terminology of~\cite{Bauer}, a {\it Loewner chordal family}~$(F_t)_{t\ge0}$
is a family of holomorphic self-maps of~$\UH$ satisfying $F_t(\UH)\subset
F_s(\UH)$ whenever $0\le s\le t$ and such that for each $t\ge0$,
$$
F_t(\zeta)=\zeta-\frac{t}\zeta+\gamma_t(\zeta),
$$
where $\gamma_t:\UH\to\Complex$ is a holomorphic function with\footnote{Here
$\angle\lim$ stands for the angular limit.}
$\angle\lim_{\zeta\to\infty}\zeta\gamma_t(\zeta)=0$. It is easy to see that
$C^1$-curves in $\UH\cup\{\infty\}$ going to~$\infty$ within a Stolz angle are
mapped by each $F_t$ onto \hbox{$C^1$-curves,} with the angles between them
at~$\infty$ being preserved. Hence for any $s\ge0$, any $k>1$ and any $k'>k$
there exist $R,R'>0$ such that ${D(R,k)\subset F_s\big(D(R',k')\big)}$, where
by $D(R,k)$ we denote the angular domain $\{\zeta:k\,\Im \zeta>|\zeta|>R\}$.
Using this fact it is not difficult to show that for any $s\ge0$ and any
$t\ge0$ the holomorphic function ${\Phi_{s,t}:=F_s^{-1}\circ F_t:\UH\to\UH}$
has a similar expansion
${\Phi_{s,t}(\zeta)=\zeta-(t-s)/\zeta+\gamma_{s,t}(\zeta)}$ with
$\angle\lim_{\zeta\to\infty}\zeta\gamma_{s,t}(\zeta)=0$. Then
$|\Phi_{s,t}(\zeta)-\zeta|\le(t-s)/\Im\zeta$, see, e.g.,
\cite[p.\,7--12]{Aleks1983} or \cite[p.\,567--568]{SMP2}. Since for each $T>0$
the family $\{F_s:s\in[0,T]\}$ is locally uniformly bounded (again use
$|F_s(\zeta)-\zeta|\le s/\Im \zeta$), this inequality leads to an estimate for
$|F_t(\zeta)-F_s(\zeta)|=|F_s(\Phi_{s,t}(\zeta))-F_s(\zeta)|$ in terms
of~$t-s$, which in turn implies that  up to the Cayley map, Loewner chordal
families defined in~\cite{Bauer} are a particular case of decreasing Loewner
chains of order~$d$, introduced in this paper. Thus Theorem~\ref{TH_f_t-G}(ii)
can be regarded an extension of~\cite[Theorem~5.3]{Bauer} to the general case,
while Theorem~\ref{TH_G-f_t}(iii) represents an extension
of~\cite[Theorem~5.6]{Bauer}.
\end{remark}

\subsection{Proof of Theorem~\ref{TH_f_t-G}}
Let us fix any $T>0$ and define
$$
h_t^T:=\left\{
\begin{array}{ll}
 f_{T-t}, & \text{if $t\in[0,T]$},\\
 \id_\UD, & \text{if $t\in(T,+\infty)$}.\\
\end{array}\right.
$$
It is easy to see that $(h^T_t)_{t\ge0}$ is an (increasing) Loewner chains of
order~$d$. By Theorem~\ref{TH_SMP-EF-LC} there exist an evolution family
$(\varphi^T_{s,t})$ of order~$d$ such that $h^T_s=h^T_t\circ\varphi^T_{s,t}$
whenever $0\le s\le t$. In particular,
\begin{equation}\label{EQ_T1}
f_{T-s}=\varphi_{s,T}^T\quad\text{for each $s\in[0,T]$}.
\end{equation}
Denote by $G_T$ the Herglotz vector field of order $d$ that corresponds to the
evolution family~$(\varphi_{s,t}^T)$ in the sense of
Theorem~\ref{TH_BCM-EF-VF}.  Then by \cite[Theorem 6.6]{BCM1}, $\partial
\varphi^T_{s,T}(z)/\partial s=-G_T(z,s)(\varphi^T_{s,T})'(z)$ for all $z\in\UD$ and
a.e. $s\in[0,T]$. Note that from the very definition of a decreasing Loewner
chain it follows easily that $F(z,t):=f_t(z)$ satisfies conditions S1\,--\,S3.
Hence using~\eqref{EQ_T1}, we may conclude, in accordance
with~Subsect.\,\ref{SS_solution}, that $F|_{\UD\times[0,T]}$ is a solution
to~\eqref{EQ_dec-PDE} with $G(z,t):=G_T(z,T-t)$ for all $z\in\UD$ and all
$t\in[0,T]$. The vector field $G$ is defined by $(f_t)$ via~\eqref{EQ_dec-PDE}
uniquely up to a null-set in $[0,T]$. That is why there exists a null-set $N_0$
and a function $G:\UD\times[0,+\infty)\to\Complex$ holomorphic in the first
(complex) variable and measurable in the second (real) variable such that for
each $n\in\Natural$, we have $G(z,t)=G_n(z,n-t)$ for all $t\in[0,n]\setminus
N_0$ and all $z\in\UD$. Clearly, such a function $G$ is a Herglotz vector field
of order~$d$. In this way $F(z,t)=f_t(z)$ becomes a solution
to~\eqref{EQ_dec-PDE} on the whole semiaxis~$[0,+\infty)$. Bearing in mind
Proposition~\ref{PR_characteristics}(ii) and the fact that by definition
$f_0=\id_\UD$, we see that we have proved assertion~(ii).

Let us now prove (i). Fix again $T>0$. Recall that by $\Delta(E)$ we denote the
set~$\{(s,t):s,t\in E, \, s\le t\}$. By construction,
\begin{equation}\label{EQ_g_phi}
g_s|_{\Omega_t}=\varphi^T_{T-t,T-s}\circ g_t\quad \text{for all
$(s,t)\in\Delta([0,T])$}.
\end{equation}
Fix any $z\in\UD$. Choose any $t_0\le t(z)$ and let $T:=[t_0]+1$, where $[x]$
stands for the integer part of~$x$. Then by~\eqref{EQ_g_phi},
$g_s(z)=\varphi^T_{T-t_0,T-s}\big(g_{t_0}\big)$ for all $s\in[0,t_0]$. It is
known, see, e.g., \cite[Theorem~3.6(iii)]{SMP_annulusI}, that for each
evolution family (i.e., for each $T>0$ in our case) there exists a null-set
$N(T)\subset [0,+\infty)$ such that for every $z\in\UD$ and every~$s\ge0$,
$(\partial/\partial t)\varphi^T_{s,t}(z)$ exists and equals
$G_T\big(\varphi^T_{s,t}(z),t\big)$ whenever $t\ge s$ and $t\not\in N(T)$. Note
that $N(T)$ depends neither on~$z$, nor on~$t$. Bearing in mind that
$(\varphi_{s,t}^T)$ is locally absolutely continuous in~$t$, we therefore
conclude that $s\mapsto g_s(z)$ is absolutely continuous on~$[0,t_0]$ and that
$(d/ds)g_s(z)$ exists and equals $-G(g_s(z),s)$ for all
$s\in[0,t_0]\setminus(N_0\cup N(T))$.

Since $t_0$ can be chosen arbitrarily close to $t(z)$, the above argument
proves assertion~(i) with $N:=N_0\cup\big(\cup_{n\in\Natural}N(n)\big)$, except
for the fact that the solution $w_z(t):=g_t(z)$ to~\eqref{EQ_dec-ini-ODE} has
no extension beyond $t=t(z)$. Assume on the contrary that $t(z)<+\infty$ and
that such an extension $w_z^*$ exists. Denote by $E_z$ its domain of
definition. Then by Proposition~\ref{PR_characteristics}(iii) and assertion
(ii) of the theorem we are proving now, $f_t(w_z^*(t))=z$ for all $t\in E_z$.
In particular, $z\in f_t(\UD)=\Omega_t$ for all $t\in E_z$. Hence $\sup E_z\le
t(z)$. This contradiction completes the proof of assertion~(i).

It remains to show that assertion (iii) holds. First of all, the fact that
$F(z,t):=F_0\big(f_t(z)\big)$ solves the initial value problem~$F(\cdot,0)=F_0$
for~\eqref{EQ_dec-PDE} follows immediately from~(ii). To prove the uniqueness
of the solution, we recall that by assertion~(i) we proved above, for any
$z\in\UD$ the function $[0,t(z))\ni z\mapsto g_t(z)$
solves~\eqref{EQ_dec-ini-ODE}. Therefore by
Proposition~\ref{PR_characteristics}(iii), $F(g_t(z),t)=F(g_0(z),0)=F_0(z)$ for
any $t\ge0$ and any $z\in\Omega_t$. The proof of~(iii) is now finished, because
$g_t=f_t^{-1}$ for all $t\ge0$.

Finally the essential uniqueness of~$G$ holds because each of the
equations~\eqref{EQ_dec-PDE} and~\eqref{EQ_dec-ini-ODE} defines $G$ uniquely up
to a null-set in $[0,+\infty)$.
\proofbox%

\subsection{Proof of Theorem~\ref{TH_G-f_t}}
Statement~(i) of this theorem is a standard fact in the theory of
Carath\'eofory ODEs, see, e.g., \cite[Sect.\,2]{SMP_annulusI}, which in
particular contains the proof of a more general statement \cite[Theorem
2.3(i)]{SMP_annulusI}.

In order to prove (ii) we fix an arbitrary $T>0$ and consider the following
Herglotz vector field of order~$d$,
$$
G_T(z,t):=\left\{
\begin{array}{ll}
G(z,T-t), & \text{if $t\in[0,T]$,}\\
0, & \text{if $t>T$}.\\
\end{array}
\right.
$$
By Theorem~\eqref{TH_BCM-EF-VF}, there exists an evolution
family~$(\varphi_{s,t}^T)$ or order $d$ such that for each $z\in \UD$ and
$s\ge0$, the function $[s,+\infty)\ni t\mapsto
w^T_{\zeta,s}(t):=\varphi_{s,t}^T(\zeta)\in\UD$ is the unique solution to the
initial value problem
\begin{equation}\label{EQ_proof_INI}
\frac{dw}{dt}=G_T(w,t),~~t\ge s,\quad w(s)=\zeta.
\end{equation}
It follows that for any $z\in\tilde\Omega_T:=\varphi^T_{0,T}(\UD)$ the unique
solution $t\mapsto w_z(t)=:g_t(z)$ to the initial value problem~\eqref{ode1} is
defined at least for all $t\in[0,T]$ and given for these $t$ by the formula
$w_z(t)=\varphi^T_{0,T-t}(\zeta)$, where
$\zeta:=\big(\varphi^T_{0,T}\big)^{-1}(z)$. Therefore, $\tilde\Omega_T\subset
\Omega_T$. On the other hand, the uniqueness of the solution
to~\eqref{EQ_proof_INI} implies that for any $z\in\Omega_T$, the restriction
$w_z|_{[0,T]}$ coincides with $[0,T]\ni t\mapsto \varphi^T_{0,T-t}(\zeta)$,
where $\zeta:=w_z(T)$. In particular, for any $z\in\Omega_T$, we have
$z=w_z(0)=\varphi^T_{0,T}(\zeta)\in\tilde\Omega_T$. Thus
$\tilde\Omega_T=\Omega_T$ and
$g_T=\big(\varphi^T_{0,T}\big)^{-1}:\tilde\Omega_t\to\UD$. Since $T>0$ is
arbitrary, this proves~(ii).

To prove~(iii), fix again an arbitrary $T>0$. By the above argument, for each
$t\in[0,T]$,
\begin{equation*}
g_t|_{\Omega_T}=\varphi^T_{0,T-t}\circ \big(\varphi^T_{0,T}\big)^{-1}=
\varphi^T_{0,T-t}\circ\big(\varphi^T_{T-t,T}\circ\varphi^T_{0,T-t}\big)^{-1}
=\varphi^{-1}_{T-t,T}|_{\Omega_T},
\end{equation*}
where we used condition~EF2 from Definition~\ref{D_BCM-EF} of an evolution
family. By the uniqueness principle for holomorphic functions this means that
$f_t=g_t^{-1}=\varphi_{T-t,T}$ for all $t\in[0,T]$. Consider the family
$$
h_t^T:=\left\{
\begin{array}{ll}
\varphi_{t,T}, &\text{if $t\in[0,T]$,}\\
\id_\UD,&        \text{if $t\in[0,T]$.}\\
\end{array}
\right.
$$
Clearly,
\begin{equation}\label{EQ_h_t-f_t}
f_t=h^T_{T-t}\quad \text{for all $t\in[0,T]$}.
\end{equation}
Note that by Remark~\ref{RM_EF-uni}, each function $\varphi_{s,t}^T$ is
univalent in~$\UD$. Hence for each $t\ge0$, $h_t^T$ is univalent in~$\UD$.
Furthermore, $\varphi^T_{s,t}=\id_\UD$ whenever $t\ge s\ge T$ because
$G_T(\cdot,t)\equiv0$ for all $t\ge T$. Taking into account EF2, we conclude
also that $\varphi^T_{s,t}=\varphi^T_{s,T}$ if $0\le s\le T\le t$. Now using
again EF2, it is easy to see that $h_t^T\circ\varphi_{s,t}^T=h_s$ whenever
$0\le s\le t$. By \cite[Lemma~3.2]{SMP}, $(h_t^T)$ is a Loewner chain of
order~$d$. Since $T>0$ can be chosen arbitrarily, in view
of~\eqref{EQ_h_t-f_t}, this implies that~$(f_t)$ is a decreasing Loewner chain
of order~$d$. As we have already mentioned in the proof of
Theorem~\ref{TH_f_t-G}, any (decreasing or increasing) Loewner chain satisfies
conditions S1\,--\,S3 in the definition of a solution to the generalized
Loewner\,--\,Kufarev PDE, formulated in Subsect.\,\ref{SS_solution}.

Let us show that S4 is also satisfied. Indeed, according to \cite[Theorem
6.6]{BCM1} for each $z\in\UD$ and for a.e. $t\in[0,T]$,
$$
\frac{\partial f_t(z)}{\partial t}=\frac{\partial\varphi^T_{T-t,T}(z)}{\partial
t}=\frac{\partial\varphi^T_{T-t,T}(z)}{\partial z}G_T(z,T-t)=\frac{\partial
f_t(z)}{\partial z}G(z,t).
$$
Again, since one can choose $T>0$ arbitrarily large, S4 holds for the whole
semiaxis $E=[0,+\infty)$. Thus $\UD\times[0,+\infty)\ni(z,t)\mapsto f_t(z)$ is
a solution to~\eqref{EQ_LK-PDE-ini}. The uniqueness of the solution is proved
in the same way as in Theorem~\ref{TH_f_t-G}. This completes the proof
of~(iii). \proofbox

\section{Reverse evolution families versus decreasing Loewner chains and Herglotz vector
fields}\label{S_REF} In this section we would like to discuss in more detail
the notion of a reverse evolution family, introduced in
Subsect.~\ref{SS_results}, see Definition~\ref{D_Rev_EF}, and its relationship
with decreasing Loewner chains and Herglotz vector fields.

\subsection{Statements of results}
The theorem below is an analogue of Theorem~\ref{TH_SMP-EF-LC}.
\begin{theorem}\label{TH_REF_DC}
For each $d\in[1,+\infty]$ The formula
\begin{equation}\label{EQ_REF_DC}
\varphi_{s,t}:=f_s^{-1}\circ f_t,\quad (s,t)\in\Delta\big([0,+\infty)\big),
\end{equation}
establishes a 1-to-1 correspondence between decreasing Loewner chains~$(f_t)$
of order~$d$ and reverse evolution families~$(\varphi_{s,t})$ of the same
order~$d$. Namely, for every decreasing Loewner chain $(f_t)$ of order~$d$ the
family $(\varphi_{s,t})_{(s,t)\in\Delta([0,+\infty))}$ defined
by~\eqref{EQ_REF_DC} is a reverse evolution family of order~$d$. Conversely,
for any reverse evolution family $(\varphi_{s,t})$ of order~$d$ the family
$\big(f_t\big)_{t\ge0}=\big(\varphi_{0,t}\big)_{t\ge0}$ is a decreasing Loewner chain of
order~$d$ satisfying equality~\eqref{EQ_REF_DC}.
\end{theorem}
In the situation described in the above theorem we will say that the decreasing
Loewner chain $(f_t)$ and the reverse evolution family $(\varphi_{s,t})$ are
{\it associated with} each other.

In Theorem~\ref{TH_G-f_t} we described the solutions to the generalized
Loewner\,--\,Kufarev ODE $dw/ds=-G(w,s)$ with the initial condition at $s=0$.
It appeared that this initial value problem generate, for a fixed Herglotz
vector field $G$ and variable initial data, the family of the inverse mappings
$g_s:=f_s^{-1}$ of some decreasing Loewner chain $(f_t)$. The following theorem
shows that if instead  we consider the initial condition {\it at the right
end-point} $w(t)=z\in\UD$, where $t>0$ and the solutions are looked for on the
interval $s\in[0,t]$, then we obtain the reverse evolution
family~$(\varphi_{s,t})$ associated with $(f_t)$. The converse statement is
true and it is also included in this theorem, which can be regarded as a
``decreasing" analogue of Theorem~\ref{TH_BCM-EF-VF}.
\begin{theorem}\label{TH_REF-EQS}
Let $d\in[1,+\infty]$. The following statements hold:
\begin{mylist}
\item[(i)] The generalized Loewner\,--\,Kufarev ODE
\begin{equation}\label{EQ_LK1-INI}
\frac{dw}{ds}=-G(w,s),~~~ s\in[0,t],\quad w(t)=z,
\end{equation}
establishes essentially a 1-to-1 correspondence between reverse evolution
families~$(\varphi_{s,t})$ of order~$d$ and Herglotz vector fields~$G$ of the
same order. Namely, given a reverse evolution family~$(\varphi_{s,t})$ of
order~$d$, there exists an essentially unique\footnote{This means ``unique up a
null-set on the $t$-axis".} Herglotz vector field~$G$ of order~$d$ such that
for each $t\ge0$ and $z\in\UD$ the function $[0,t]\ni s\mapsto
w(s):=\varphi_{s,t}(z)$ solves the initial value problem~\eqref{EQ_LK1-INI}.
Conversely, given a Herglotz vector field~$G$ of order~$d$, for every~${t>0}$
and every~$z\in\UD$ the initial value problem~\eqref{EQ_LK1-INI} has a unique
solution $s\mapsto w=w_{z,t}(s)$ defined for all $s\in[0,t]$ and the formula
$\varphi_{s,t}(z):=w_{z,t}(s)$ for all $z\in\UD$ and all
$(s,t)\in\Delta\big([0,+\infty)\big)$ defines a reverse evolution
family~$(\varphi_{s,t})$ of order~$d$.

\item[(ii)] Let $(\varphi_{s,t})$ and $G$ be as in statement~$(i)$ above. Then a family
$(f_t)_{t\ge0}$ of holomorphic functions in~$\UD$ is the decreasing Loewner
chain associated with~$(\varphi_{s,t})$ if and only if the function
$F:\UD\times[0,+\infty)\to\Complex$, defined by $F(z,t):=f_t(z)$ for all
$t\ge0$ and all $z\in\UD$ is a solution to the following initial value problem
for the Loewner\,--\,Kufarev~PDE
\begin{equation}\label{EQ_LK2-DEC}
\frac{\partial F}{\partial t}=\frac{\partial F}{\partial z} G(z,t),\quad
t\ge0,\quad F(\cdot,0)=\id_\UD.
\end{equation}
\end{mylist}
\end{theorem}
In the proofs of the above theorems we will use the following statement of
independent interest, which provides several alternatives for condition REF3 in
Definition~\ref{D_Rev_EF} equivalent  to EF3 under the assumption that EF1 and
EF2 are satisfied.
\begin{proposition}\label{PR_evol_equiv}
Let $d\in[1,+\infty]$. Suppose that a family $(\varphi_{s,t})_{0\le s\le t}$ of
holomorphic self-maps of~$\UD$ satisfies conditions REF1 and REF2 from
Definition\,\ref{D_Rev_EF}. Then the following conditions are equivalent:
\begin{mylist}
\item[(i)] for each $T>0$ there exist two distinct points $\zeta_j\in\UD$, $j=1,2$,
and a non-negative function $k_T\in L^{d}([0,T],\mathbb{R})$ such that for
$j=1,2$,
\begin{equation}\label{EQ_qqPR}
|\varphi_{s,t}(\zeta_j)-\varphi_{s,u}(\zeta_j)|\leq\int_{u}^{t}k_{T}(\xi)d\xi
\end{equation}
whenever $0\le s\le u\le t\le T$.

\item[(ii)] $(\varphi_{s,t})$ satisfies condition~REF3, i.e. $(\varphi_{s,t})$
is a reverse evolution family of order~$d$.

\item[(iii)] for any $T>0$ and any compact set $K\subset\mathbb{D}$  there exists a
non-negative function $k_{K,T}\in L^{d}([0,T],\mathbb{R})$ such that for all
$z\in K$,
\begin{equation}\label{EQ_K1}
|\varphi_{s,t}(z)-\varphi_{s,u}(z)|\leq\int_{u}^{t}k_{K,T}(\xi)d\xi
\end{equation}
whenever $0\le s\le u\le t\le T$.

\item[(iv)] for any $T>0$ there exists an evolution family~$(\varphi_{s,t}^T)$
of order~$d$ such that ${\varphi_{s,t}=\varphi_{T-t,T-s}^T}$ for all
$(s,t)\in\Delta\big([0,T]\big)$.

\item[(v)] for any $T>0$ and any compact set $K\subset\mathbb{D}$  there exists a
non-negative function $k_{K,T}\in L^{d}([0,T],\mathbb{R})$ such that for all
$z\in K$,
\begin{equation}\label{EQ_regularity_in_s}
|\varphi_{s,t}(z)-\varphi_{u,t}(z)|\leq\int_{s}^{u}k_{K,T}(\xi)d\xi
\end{equation}
whenever $0\le s\le u\le t\le T$.

\end{mylist}

In particular, every element of any reverse evolution family is a univalent
function in~$\UD$.
\end{proposition}

\begin{remark}
Before stating the proofs we would like to mention that the families
$\{B(a,b;\cdot)\}_{0\le a\le b}$ introduced in~\cite{Bauer} and referred there
to as {\it semigroups associated with chordal Loewner families} are a special
case of reverse evolution families of order~$d=+\infty$. Accordingly, the
results we have stated in this section extend Theorems 5.4 and 5.5
from~\cite{Bauer}.
\end{remark}

The equivalence of (i), (iii), and (iv) in Proposition~\ref{PR_evol_equiv},
being formulated in the ``increasing" context, immediately leads to the
following statement of some independent interest.
\begin{corollary}
Let $d\in[1,+\infty]$. Suppose that a family $(\varphi_{s,t})_{0\le s\le t}$ of
holomorphic self-maps of~$\UD$ satisfies conditions EF1 and EF2 from
Definition\,\ref{D_BCM-EF}. Then the following conditions are equivalent:
\begin{mylist}
\item[(i)]  $(\varphi_{s,t})$ satisfies condition~EF3, i.e. $(\varphi_{s,t})$
is a evolution family of order~$d$.

\item[(ii)] for any $T>0$ there exist two distinct points $\zeta_j\in\UD$, $j=1,2$,
and a non-negative function $k_T\in L^{d}([0,T],\mathbb{R})$ such that for
$j=1,2$,
\begin{equation*}
|\varphi_{s,t}(\zeta_j)-\varphi_{u,t}(\zeta_j)|\leq\int_{s}^{u}k_{T}(\xi)d\xi
\end{equation*}
whenever $0\le s\le u\le t\le T$.

\item[(iii)] for any $T>0$ and any compact set $K\subset\mathbb{D}$  there exists a
non-negative function $k_{K,T}\in L^{d}([0,T],\mathbb{R})$ such that for all
$z\in K$,
\begin{equation*}
|\varphi_{s,t}(z)-\varphi_{u,t}(z)|\leq\int_{s}^{u}k_{K,T}(\xi)d\xi
\end{equation*}
whenever $0\le s\le u\le t\le T$.

\end{mylist}
\end{corollary}

\subsection{Proof of Proposition~\ref{PR_evol_equiv}}
Recall that by $\rho_\UD(\cdot,\cdot)$ we denote the pseudohyperbolic distance
in $\UD$. Note that $$|w-z|\le2\rho_\UD(z,w)$$ for any $z,w\in\UD$. Hence
combining conditions REF1 and REF2, the Schwarz\,--\,Pick lemma, and
Corollary~\ref{C_pseudo_symm}, we deduce that for any compact set $K\subset\UD$
and any two distinct points $\zeta_1,\zeta_2\in\UD$ there exists
${M=M(K,\zeta_1,\zeta_2)>0}$ such that
\begin{multline}\label{EQ_step1}
\big|\varphi_{s,t}(z)-\varphi_{s,u}(z)\big|\le2\rho_\UD\big(\varphi_{s,t}(z),\varphi_{s,u}(z)\big)=
2\rho_\UD\big(\varphi_{s,u}(\varphi_{u,t}(z)),\varphi_{s,u}(z)\big)\\\le
2\rho_\UD\big(\varphi_{u,t}(z),z\big)\le
M\big(|\varphi_{u,t}(\zeta_1)-\zeta_1|+|\varphi_{u,t}(\zeta_2)-\zeta_2|\big)
\end{multline}
for all $z\in K$ and all $s,u,t\ge0$ such that $s\le u\le t$.

Obviously, $\text{(iii)}\Longrightarrow\text{(ii)}\Longrightarrow\text{(i)}$.
Moreover, substituting $(u,u,t)$ for $(s,u,t)$ in REF3 and bearing in
mind~\eqref{EQ_step1}, one can easily see that (i) implies (iii).

Let us prove that $\text{(iv)}\Longrightarrow\text{(iii)}$.
Fix any $T>0$ and a compact set $K\subset\UD$. Choose any two distinct points
$z_1,z_2\in\UD$.  Denote $w_j(s):=\varphi_{s,T}(\zeta_j)$ for all $s\in[0,T]$
and $j=1,2$. From REF2 it follows that for any $(u,t)\in\Delta\big([0,T]\big)$
and $j=1,2$,
\begin{equation}\label{EQ_j}
\big|\varphi_{u,t}\big(w_j(t)\big)-w_j(t)\big|=\big|w_j(u)-w_j(t)\big|.
\end{equation}
By (iv) there exists an evolution family~$(\varphi_{s,t}^T)$ of order~$d$ such
that $\varphi_{s,t}=\varphi_{T-t,T-s}^T$ for all
$(s,t)\in\Delta\big([0,T]\big)$. In particular, $w_j(s)=\varphi^T_{0,T-s}(z_j)$
for all $s\in[0,T]$ and $j=1,2$. Hence the functions $w_1$ and $w_2$ are of
class $AC^d$ on $[0,T]$. Moreover, with all $\varphi_{s,t}^T$'s being univalent
by Remark~\ref{RM_EF-uni}, these two functions do not share common values. In
particular, it follows that there exist $R\in(0,1)$ and $\rho>0$ such that
$|w_j(t)|\le R$, $j=1,2$, and $\rho_\UD\big(w_1(t),w_2(t)\big)\ge \rho$ for all
$t\in[0,T]$. Apply now Corollary~\ref{C_pseudo_symm} with $w_j(t)$ substituted
for~$\zeta_j$ to deduce from~\eqref{EQ_j} that there is a constant
$M_1=M_1(K,T,z_1,z_2)>0$ such that
\begin{equation}\label{EQ_byREF3j}
\rho_\UD\big(\varphi_{u,t}(z),z\big)\le
M_1\big(|w_1(t)-w_1(u)|+|w_2(t)-w_2(u)|\big)
\end{equation}
for all $(u,t)\in\Delta\big([0,T]\big)$ and all $z\in K$.
Inequalities~\eqref{EQ_step1} and~\eqref{EQ_byREF3j} imply together that
\begin{equation*}
\big|\varphi_{s,t}(z)-\varphi_{s,u}(z)\big|\le
2M_1\big(|w_1(t)-w_1(u)|+|w_2(t)-w_2(u)|\big)
\end{equation*}
for all $(u,t)\in\Delta\big([0,T]\big)$ and all $z\in K$. Recalling that
$w_1,w_2\in AC^d\big([0,T],\UD\big)$, we see that assertion~(iii) holds with
$$ k_{K,T}(\xi):=2M_1\big(|
w_1'(\xi)|+|w_2'(\xi)|\big).$$

Now assume (i) and let us
prove (v). To this end we again fix any $T>0$ and any compact set
$K\subset\UD$. Let us show first that from (i) it follows that there exists
$r\in(0,1)$ such that
\begin{equation}\label{EQ_rr}
\varphi_{u,t}(K)\subset r\UD\quad\text{ for all
$(u,t)\in\Delta\big([0,T]\big)$.}
\end{equation}
Assume on the contrary that there exists a sequence
$\big((u_n,t_n)\big)\subset\Delta\big([0,T]\big)$ such that
$\sup|\varphi_{u_n,t_n}(K)|\to 1$ as $n\to+\infty$. Without loss of generality
we may also assume that $(u_n,t_n)\to(u_0,t_0)$ as $n\to+\infty$ for some
$(u_0,t_0)\in\Delta\big([0,T]\big)$.

Note that by (i),
\begin{equation}\label{EQ_from(i)}
|\varphi_{u,t}(\zeta_j)-\zeta_j|\le\omega(t-u)\quad\text{for $j=1,2$ and for
any $(u,t)\in\Delta\big([0,T]\big)$,}
\end{equation}
where  $\omega(\cdot)$ stands for the modulus of continuity of $[0,T]\ni
t\mapsto\int_{0}^tk_T(\xi)d\xi$. Hence from~\eqref{EQ_step1} it follows that
$\varphi_{u_n,t_n}-\varphi_{u_n,t_0}\to0$ uniformly on $K$ as $n\to+\infty$.
Therefore
\begin{equation}\label{EQ_to1}
\sup|\varphi_{u_n,t_0}(K)|\to 1\quad\text{ as $n\to+\infty$.}
\end{equation}
Now choose $\delta>0$ in such a way that
$\omega(2\delta)<\varepsilon_0:=\big(1-|\zeta_1|\big)/2$, and let
$n_0\in\Natural$ be such that $|u_n-u_0|<\delta$ for all natural $n\ge n_0$.
Then $|\varphi_{u_n,u_0+\delta}(\zeta_1)|<1-\varepsilon_0$ for all $n>n_0$.
Bearing in mind that $\varphi_{u_n,u_0+\delta}(\UD)\subset\UD$ for all
$n\in\Natural$, we conclude that $\{\varphi_{u_n,u_0+\delta}:n>n_0\}$ is
relatively compact in $\Hol(\UD,\UD)$. In view of the fact that by REF2,
$\varphi_{u_n,t_0}=\varphi_{u_n, u_0+\delta}\circ\varphi_{u_0+\delta,t_0}$ for
all $n>n_0$, this contradicts~\eqref{EQ_to1} and thus proves~\eqref{EQ_rr} for
some $r\in(0,1)$ depending on the compact set $K$ and $T>0$.

Now taking advantage of REF2, we get
$$%
\big|\varphi_{s,t}(z)-\varphi_{u,t}(z)\big|=|\varphi_{s,u}(\zeta)-\zeta|,\quad
\zeta:=\varphi_{u,t}(z)
$$%
for any $z\in\UD$ and any $s,u,t\ge0$ such that $s\le u\le t$.
Using~\eqref{EQ_rr} we apply Corollary~\ref{C_pseudo_symm} to conclude that
there exists $M_2=M_2(T,K,\zeta_1,\zeta_2)>0$ such that
$$%
\big|\varphi_{s,t}(z)-\varphi_{u,t}(z)\big|\le
M_2\big(|\varphi_{u,t}(\zeta_1)-\zeta_1|+|\varphi_{u,t}(\zeta_2)-\zeta_2|\big)
$$%
whenever $z\in K$ and $0\le s\le u\le t\le T$. Now it follows easily
from~\eqref{EQ_qqPR} with $(u,u,t)$ substituted for $(s,u,t)$ that there exists
a non-negative function $\hat k_{K,T}:=M_2 k_T\in L^d\big([0,T], \Real\big)$
such that for all $z\in K$,
\begin{equation}\label{EQ_hatK}
|\varphi_{s,t}(z)-\varphi_{u,t}(z)|\leq\int_{s}^{u}\hat k_{K,T}(\xi)d\xi
\end{equation}
whenever $0\le s\le u\le t\le T$.

To complete the proof assume (v) and let us prove (iv).
Write
\begin{equation}\label{EQ_REF-EF}
\varphi^T_{s,t}:=\left\{
\begin{array}{ll}
\varphi_{T-t,T-s},&\text{if $0\le s\le t\le T$},\\
\varphi_{0,T-s},&\text{if $0\le s\le T\le t$},\\
\id_{\UD},&\text{if $T\le s\le t$}.\\
\end{array}
\right.
\end{equation}
In other words, for all $s\ge0$ and all $t\ge s$, we have
$\varphi_{s,t}^T=\varphi_{\tau(t),\tau(s)}$, where $\tau(t)$ stands for
${\max\{0,T-t\}}$.

It is easy to see that $(\varphi^T_{s,t})$ satisfies conditions EF1 and EF2 in
Definition~\ref{D_BCM-EF}. Moreover, combining \eqref{EQ_hatK}
and~\eqref{EQ_REF-EF},  for all $z\in\UD$ and all $s,u,t$ such that $0\le s\le
u\le t$ we have
\begin{equation*}
\big|\varphi_{s,t}^T(z)-\varphi_{s,u}^T(z)\big|\le
\big|\varphi_{\tau(t),\tau(s)}(z)-\varphi_{\tau(u),\tau(s)}(z)\big|\le\int_{u}^t\tilde
k_{K,T}(\xi),
\end{equation*}
where $\tilde k_{K,T}(\xi):=k_{K,T}(T-\xi)$ for all $\xi\in[0,T]$ and
$\tilde k_{K,T}(\xi):=0$ for all $\xi>T$.  It follows that $(\varphi_{s,t}^T)$
satisfies EF3. Thus (iv) is true.

Finally, using the implication $\text{(ii)}\Longrightarrow\text{(iv)}$ we can
easily conclude that all the elements of any reverse evolution family is a
univalent function in~$\UD$, since by Remark~\ref{RM_EF-uni} this is the case
for elements of evolution families. Now the proof of
Proposition~\ref{PR_evol_equiv} is complete. \proofbox

\subsection{Proof of Theorem~\ref{TH_REF_DC}} Assume first that we are
given a decreasing Loewner chain~$(f_t)$ of order $d$. Fix any $T>0$. Then the
family $(h_t^T)_{t\ge0}$ defined by
\begin{equation}\label{EQ_h_t}
h_t^T:=\left\{
\begin{array}{ll}
 f_{T-t}, & \text{if $t\in[0,T]$},\\
 \id_\UD, & \text{if $t\in(T,+\infty)$},\\
\end{array}\right.
\end{equation}
is an (increasing) Loewner chain. Hence by Theorem~\ref{TH_SMP-EF-LC}, the
formula $\varphi_{s,t}^T:=h_t^{-1}\circ h_s$ for all $s\ge0$ and all $t\ge s$
is an evolution family of order~$d$. Note that for all
$(s,t)\in\Delta\big([0,T]\big)$ we have $\varphi_{s,t}=\varphi^T_{T-t,T-s}$,
where $(\varphi_{s,t})$ is defined by~\eqref{EQ_REF_DC}. Thus, bearing in mind
that $T>0$ can be chosen arbitrarily, we conclude that by
Proposition~\ref{PR_evol_equiv}, $(\varphi_{s,t})$ is a reverse evolution
family of order~$d$.

To prove the converse statement we assume now that $(\varphi_{s,t})$ a reverse
evolution family of order $d$. We have to show that $f_t:=\varphi_{0,t}$ is a
decreasing Loewner chain of order~$d$ satisfying~\eqref{EQ_REF_DC}. Recall that
by Proposition~\ref{PR_evol_equiv}, the functions $\varphi_{s,t}$ are univalent
in~$\UD$. It follows that $(f_t)$ satisfies LC1. By the same reason, REF2
implies~\eqref{EQ_REF_DC}. Moreover, REF2 implies also that $(f_t)$ satisfies
LC2. Finally, LC3 follows immediately from assertion~(iii) of
Proposition~\ref{PR_evol_equiv}. The proof is complete.\proofbox

\subsection{Proof of Theorem~\ref{TH_REF-EQS}}
We start with the proof of~(i). Assume at first that we are given a reverse
evolution family~$(\varphi_{s,t})$ of order~$d$ and let us prove that there
exists a Herglotz vector field of order~$d$ that generate $(\varphi_{s,t})$
via~\eqref{EQ_LK1-INI}. To this end we fix $T>0$ and apply
Proposition~\ref{PR_evol_equiv}, according to which there exists an evolution
family~$(\varphi_{s,t}^T)$ of order~$d$ such that
$\varphi_{s,t}=\varphi_{T-t,T-s}^T$ whenever $0\le s\le t\le T$. In turn,
according to Theorem~\ref{TH_BCM-EF-VF} there exists a Herglotz vector field
$G_T$ of order~$d$ such that for any $s\ge0$ and any $z\in\UD$ the function
$[s,+\infty)\ni t\mapsto w(t):=\varphi_{s,t}^T(z)$ is the unique solution to
the equation $dw/dt=G_T(w,t)$, $t\ge s$, with the initial condition~$w(s)=z$.
It follows that for each $t\in(0,T]$ and each $z\in\UD$ the function $[0,t]\ni
s\mapsto w_{z,t}(s):=\varphi_{s,t}(z)$ is the unique solution to the initial
value problem
\begin{equation}\label{EQ_ini_REF-EF}
\frac{dw}{ds}=-G_T(w,T-s),~~s\in[0,t],\quad w(t)=z.
\end{equation}
Since the functions $(z,s)\mapsto w_{z,t}(s)$ define $G_T$
via~\eqref{EQ_ini_REF-EF} uniquely up to a null-set in $[0,T]$, there exists a
function $G:\UD\times[0,+\infty)\to\Complex$ such that for each $n\in\Natural$,
$G(\cdot,t)=G_n(\cdot,n-t)$ for a.e. $t\in[0,n]$. Clearly, $G$ is the desired
Herglotz vector field of order~$d$.

Now we pass to the converse statement. So assume that we are given a Herglotz
vector field~$G$ of order~$d$ and let us prove that it generates a reverse
evolution family $(\varphi_{s,t})$ of order~$d$. Again fix any $T>0$. Arguing
as in the proof of Theorem~\ref{TH_G-f_t}(ii), one can construct an evolution
family~$(\varphi_{s,t}^T)$ of order~$d$ such that for each $z\in\UD$ and each
$t\in(0,T]$ the function $[0,t]\ni s\mapsto w_{z,t}(s):=\varphi_{T-t,T-s}^T(z)$
is the unique solution to the initial value problem~\eqref{EQ_LK1-INI}. Note
that by uniqueness of the solution, $w_{z,t}$ does not depend on $T$. Hence
there exists a unique family $(\varphi_{s,t})_{0\le s\le t}$ of holomorphic
functions $\varphi_{s,t}:\UD\to\UD$ such that
\begin{equation}\label{EQ_defREFviaEF}
\varphi_{s,t}(z)=w_{z,t}(s)=\varphi^T_{T-t,T-s}(z)
\end{equation}
for all $z\in\UD$, all $(s,t)\in\Delta\big([0,+\infty)\big)$, and all $T\ge t$.

Take now any $u,s,t\ge0$ such that $u\le s\le t$. Choose any $T\ge t$. Then
combining~\eqref{EQ_defREFviaEF} with conditions EF1 and EF2 for
$(\varphi_{s,t}^T)$ we easily obtain REF1 and REF2. Furthermore, applying now
Proposition~\ref{PR_evol_equiv} we conclude that REF3 holds as well. Hence
$(\varphi_{s,t})$ is a reverse evolution family of order~$d$. To complete the
proof of~(i) it remains to recall that by~\eqref{EQ_defREFviaEF}, for each
$t\ge0$ and each $z\in\UD$, $[0,t]\ni s\mapsto\varphi_{s,t}(z)$ is the unique
solution to~\eqref{EQ_LK1-INI}.

Now let us prove (ii). First let us assume that $F(z,t):=f_t(z)$
solves~\eqref{EQ_LK2-DEC}. Then by Theorem~\ref{TH_G-f_t}, $(f_t)$ is a
decreasing Loewner chain of order $d$, with the functions
$w_\zeta(s):=f_s^{-1}(\zeta)$, where $\zeta\in\UD$, being solutions, on their
domains of definition, to the equation $dw/ds=G(w,s)$. For each $z\in\UD$ and
each $t>0$, set $\zeta:=\zeta(z,t)=f_t(z)$. Then by condition DC2 in
Definition~\ref{D_decLC}, $w_{\zeta(z,t)}$ is defined for all $s\in[0,t]$.
Moreover, it satisfies the initial condition~$w_{\zeta(z,t)}(t)=z$. By the
uniqueness of the solution to~\eqref{EQ_LK1-INI}, it follows that
$w_{\zeta(z,t)}(s)=\varphi_{s,t}(z)$ for all $z\in\UD$ and all
$(s,t)\in\Delta\big([0,+\infty)\big)$. On the other hand, by construction
$w_{\zeta(z,t)(s)}=f_s^{-1}\circ f_t(z)$ for all such $z$, $s$ and $t$. Thus
the decreasing Loewner chain is associated with the reverse evolution family
$(\varphi_{s,t})$.

It remains to prove the converse statement. So we assume that $(f_t)$ is the decreasing
Loewner chain of order~$d$ associated with~$(\varphi_{s,t})$. We have to show
that $F(z,t):=f_t(z)$ solves~\eqref{EQ_LK2-DEC}. By Theorem~\ref{TH_f_t-G},
the function $F$ is a solution to the generalized Loewner\,--\,Kufarev PDE
$$
\frac{\partial F(z,t)}{\partial t}=\frac{\partial F(z,t)}{\partial
z}G^*(z,t),\quad t\ge0,~~z\in\UD,
$$
for some Herglotz vector field $G^*$, which {\it a priori} can be different
from $G$. We have to prove that actually $G^*$ and $G$ coincide aside a null
set on the $t$-axis. To this end consider the reverse evolution
family~$(\varphi^*_{s,t})$ generated in the sense of statement~(i) of the
theorem we are proving by the Herglotz vector field~$G^*$. By what we have
already showed, $(f_t)$ must be the Loewner chain associated
with~$(\varphi^*_{s,t})$. By the very definition, this means that
$(\varphi^*_{s,t})=(\varphi_{s,t})$. Now the uniqueness of the Herglotz vector
field in statement~(i) implies that $G$ and $G^*$ essentially coincide. This
completes the proof. \proofbox

\section{Two-point characterization of Loewner chains}\label{S2}
The regularity of the $L^d$-Loewner chains, both increasing and decreasing,
w.r.t. the time parameter~$t$ is described by (literally identical) conditions
LC3 and DC3, requiring that the Loewner chain, considered as a mapping
$[0,+\infty)\ni t\mapsto f_t\in\Hol(\UD,\Complex)$, must be locally absolutely
continuous. At the same time in the classical theory, both in chordal and
radial variants, the regularity w.r.t. $t$ is achieved by controlling a unique
increasing parameter, such as $f'_t(0)$ in the radial case.
Theorem~\ref{TH_LC}, which we are going to prove in this section, states that
condition LC3 can replaced by an {\it a priori} weaker condition in the spirit
of the classical theory. This theorem has an immediate consequence for
decreasing Loewner chains.

\begin{corollary}\label{C_wDC}
Let $(f_t)_{t\ge0}$ be a family of functions satisfying conditions DC1 and DC2
from Definition~\ref{D_decLC} and let $d\in[1,+\infty]$. Then $(f_t)$ is a
decreasing Loewner chain of order~$d$ if and only if for every $T>0$ there
exist two distinct points $\zeta_1,\zeta_2\in\UD$ such that the functions
$t\mapsto w_j(t):=f_t(\zeta_j)$ belong for $j=1,2$ to the class
$AC^d\big([0,T],\Complex\big)$.
\end{corollary}
Since Corollary~\ref{C_wDC} follows directly from Theorem~\ref{TH_LC} by means
of the change of variable $t\mapsto T-t$, we omit the proof.
\begin{proof}[{\bf Proof of Theorem~\ref{TH_LC}}]
The implication~$\text{LC3}\Longrightarrow\text{LC3w}$ is obvious. So assume
that $(f_t)$ satisfies LC1, LC2, and LC3w and let us prove that $(f_t)$ is a
Loewner chain of order~$d$.

Fix any $T>0$. For $(s,t)\in\Delta([0,T])$ consider
$\varphi_{s,t}:=f_t^{-1}\circ f_s:\UD\to\UD$. Since  ${f_t(\UD)\subset
f_T(\UD)}$ for all $t\in[0,T]$ and the map $[0,T]\ni t\mapsto f_t(\zeta_1)$ is
continuous, the set $\{f_t:t\in[0,T]\}$ is relatively compact in
$\Hol(\UD,\Complex)$. Therefore applying Lemma~\ref{LM_phi-nepr} for
$K:=\{\zeta_1,\zeta_2\}$, $z_2=z_2(s,t):=\varphi_{s,t}(z_1)$, $z_1:=\zeta_j$,
we get
\begin{equation}\label{EQ_phi-z0}
|\varphi_{s,t}(\zeta_j)-\zeta_j|\le M |f_t(\zeta_j)-f_s(\zeta_j)|,\quad j=1,2,
\end{equation}
for all $(s,t)\in\Delta([0,T])$, where $M:=M_{\{\zeta_1,\zeta_2\}}>0$ does not
depend on~$s$ and~$t$.

Combining \eqref{EQ_phi-z0} with Corollary~\ref{C_pseudo_symm} for
$\varphi:=\varphi_{s,t}$ and taking into account that
$$|w-z|\le2\rho_\UD(w,z)$$ for all $z,w\in\UD$, we conclude that for each
$r\in(0,1)$,
\begin{multline}\label{EQ_phi-K}
\big|\varphi_{s,t}(\zeta)-\zeta\big|\le 2\tilde
C(r,R_0,\rho_0)\big(|\varphi_{s,t}(\zeta_1)-\zeta_1|+|\varphi_{s,t}(\zeta_2)-\zeta_2|\big)\\\le
C_0\big(|f_t(\zeta_1)-f_s(\zeta_1)|+|f_t(\zeta_2)-f_s(\zeta_2)|\big), \qquad \qquad
\end{multline}
whenever $|\zeta|\le r$ and $(s,t)\in\Delta([0,T])$, where $C_0:=2M \tilde
C(r,R_0,\rho_0),$ $R_0:=\max\{|\zeta_1|,|\zeta_2|\},$ and
$\rho_0:=\rho_\UD(\zeta_2,\zeta_1)$.

Since the mappings $[0,T]\ni t\mapsto f_t(\zeta_j)$, $j=1,2$,  are continuous,
inequality~\eqref{EQ_phi-K} implies that for each $r\in(0,1)$ there exists
$\delta>0$ such that $|\varphi_{s,t}(\zeta)|\le r':=(1+r)/2$ for all $\zeta$
with $|\zeta|\le r$ and all $(s,t)\in\Delta([0,T])$ with $t<s+\delta$. Recall
that $\{f_t:t\in[0,T]\}$ is relatively compact in $\Hol(\UD,\Complex)$. It
follows that $|f_t'(z)|$ is uniformly bounded on the disk $\{z:|z|\le r'\}$. As
a result, from~\eqref{EQ_phi-K} we get
\begin{multline}\label{EQ_ftfs}
|f_t(\zeta)-f_s(\zeta)|\le M_1|\varphi_{s,t}(\zeta)-\zeta|\le
C(r)\big(|f_t(\zeta_1)-f_s(\zeta_1)|+|f_t(\zeta_2)-f_s(\zeta_2)|\big),
\end{multline}
for any $\zeta$ with $|\zeta|\le r$ and any $(s,t)\in\Delta([0,T])$ with
$t<s+\delta$, where $C(r):=M_1C_0$ and  $M_1=M_1(r'):=\sup\big\{|f_t'(z)|:|z|\le r',~t\in[0,T]\big\}$.

Let $K$ be any compact set in~$\UD$. Choose $r\in(0,1)$ in the above argument
in such a way that $|\zeta|\le r$ for all $\zeta\in K$. Set
$k_{K,T}(\xi):=C(r)\big(k_1(\xi)+k_2(\xi)\big)$ for all $\xi\in[0,T]$, where
$k_j(t):=|df_t(\zeta_j)/dt|$ for $j=1,2$ and all $t\in[0,T]$. Obviously,
$k_{K,T}\ge0$ and belongs to $L^d_{\rm loc}\big([0,T],\Real\big)$.
From~\eqref{EQ_ftfs} it follows that whenever $0\le s\le t\le T$, $t<
s+\delta$, and $\zeta\in K$, we will have
\begin{equation}\label{EQ_LC3}
\big|f_t(\zeta)-f_s(\zeta)\big|\le\int_{s}^tk_{K,T}(\xi)d\xi.
\end{equation}
Using the triangle inequality in the right-hand side of~\eqref{EQ_LC3} and the
additivity of the integral in its left-hand side, it is easy to remove the
restriction $t<s+\delta$. Hence $(f_t)$ satisfies LC3 and thus it is Loewner
chain of order $d$.
\end{proof}


\begin{thebibliography}{99}



\bibitem{Aleksandrov} I.A. Aleksandrov,
Parametric continuations in the theory of univalent functions (Russian), Izdat.
``Nauka'', Moscow, 1976. \hbox{MR0480952 (58 \#1099)}

\bibitem{Aleks1983}I.A. Aleksandrov, S.T. Aleksandrov\ and\ V.V. Sobolev,  \textit{ Extremal properties of mappings of a half
plane into itself}, in  Complex analysis (Warsaw, 1979), 7--32, PWN,
Warsaw. MR0737747 (85h:30032)


\bibitem{AleksST} S.T. Aleksandrov,  \textit{ Parametric representation of functions univalent in the half plane,} in  Extremal
problems of the theory of functions, 3--10, Tomsk. Gos. Univ., Tomsk, 1979.
\hbox{MR0767462 (85i:30035)}

\bibitem{AleksSTSob} S.T. Aleksandrov, V.V. Sobolev,  \textit{ Extremal problems in some classes of functions, univalent
in the half plane, having a finite angular residue at infinit,} Siberian Math.
J. {\bf 27}(1986), no.\,2, 145--154. Translation from  Sibirsk. Mat. Zh. {\bf
27}(1986), no.\,2, 3--13. \hbox{MR0890295 (88j:30049)}

\bibitem{AbstractLoew} L. Arosio, F. Bracci, H. Hamada, G. Kohr,
\textit{ An abstract approach to Loewner chains}. To appear in  J. Anal. Math.

\bibitem{Bauer} R.O. Bauer, \textit{ Chordal Loewner families
and univalent Cauchy transforms},  J. Math. Anal. Appl. \textbf{302} (2005),
484-501.


\bibitem{BCM1} F. Bracci, M.D. Contreras, and S.
D\'{\i}az-Madrigal, \textit{ Evolution Families and the Loewner Equation I: the
unit disk}. To appear in J. Reine Angew. Math.

\bibitem{BCM2} F. Bracci, M.D. Contreras, and S.
D\'{\i}az-Madrigal, \textit{ Evolution Families and the Loewner Equation II:
complex hyperbolic manifolds.}  Math. Ann., \textbf{344} (2009), 947--962.










\bibitem{SMP} M.\,D. Contreras, S.~D\'\i az-Madrigal, and P.~Gumenyuk, \textit{ Loewner chains in the unit disk}. Rev. Mat. Iberoam. \textbf{26} (2010), 975--1012;


\bibitem{SMP2} M. D. Contreras, S. D\'\i az-Madrigal\ and\ P. Gumenyuk, \textit{ Geometry behind chordal
Loewner chains}, Complex Anal. Oper. Theory {\bf 4} (2010), no.~3, 541--587.
MR2719792 (2011h:30037)

\bibitem{SMP_annulusI} M. D. Contreras, S. D\'\i az-Madrigal\ and\ P. Gumenyuk,
\textit{ Loewner Theory in annulus I: evolution families and differential equations}. To appear in
Trans. Amer. Math. Soc., arXiv:1011.4253v1 [math.CV].

\bibitem{SMP_annulusII} M. D. Contreras, S. D\'\i az-Madrigal\ and\ P. Gumenyuk, \textit{
Loewner Theory in annulus II: Loewner chains}, Analysis and Mathematical Physics \textbf{ 1}
(2011), 351--385.



\bibitem{Dubovikov} D.A. Dubovikov, \textit{ An analog of the Lo\"ewner equation for mappings of strip,}
Izv. Vyssh. Uchebn. Zaved. Mat. {\bf 2007}, no. 8, 77--80 (Russian);
translation in Russian Math. (Iz. VUZ) {\bf 51} (2007), no.~8, 74--77.
MR2396110 (2008m:30017)

\bibitem{Duren} P.L. Duren,  Univalent Functions,
Springer, New York, 1983. \hbox{MR0708494 (85j:30034)}



\bibitem{Filippov} A. F. Filippov, Differential equations with discontinuous right-hand sides,
 ``Nauka'', Moscow, 1985. MR0790682 (87f:34002) (Russian); translation to English:
Kluwer Acad. Publ., Dordrecht, 1988. MR1028776 (90i:34002)


\bibitem{Goryainov}V.V. Goryainov, \textit{ Semigroups of conformal mappings,}
Mat. Sb. (N.S.) {\bf 129(171)} (1986), no.~4, 451--472 (Russian); translation
in Math. USSR Sbornik \textbf{57} (1987), 463--483.

\bibitem{Goryainov1996}
V.V. Goryajnov, \textit{ Evolution families of analytic functions and time-inhomogeneous
Markov branching processes}, Dokl. Akad. Nauk \textbf{347}(1996), No.\,6,
729--731; translation in Dokl. Math. \textbf{53}(1996), No.\,2, 256--258.

\bibitem{GoryainovTalk} V.V. Goryainov, \textit{ The Loewner\,--\,Kufarev
equation and extremal problems for conformal mappings with fixed points on the
boundary} (in Russian). Talk at the International Workshop on Complex Analysis
and its Applications, Steklov Mathematical Institute, Moscow, December 26,
2011. The slides and videorecord are available at\\
\verb+http://www.mathnet.ru/php/presentation.phtml?option_lang=eng&presentid=3999+

\bibitem{Goryainov-Ba}V.V. Goryainov and I. Ba, \textit{ Semigroups of
conformal mappings of the upper half-plane into itself with hydrodynamic
normalization at infinity,} Ukrainian Math. J. \textbf{44} (1992), 1209--1217.

\bibitem{Goryainov-Kudryavtseva} V.V. Goryainov, O.S. Kudryavtseva, \textit{One-parameter semigroups of
analytic functions, fixed points and the Koenigs function}, Mat. Sb.
\textbf{202}(2011), No.\,7, 43--74 (Russian); translation in Sbornik:
Mathematics, textbf{202}(2011), No.\,7, 971--1000.


\bibitem{Zhora}G. Ivanov, D. Prokhorov, A. Vasil'ev, \textit{Non-slit and singular solutions to the
L\"owner equation}, to appear in Bull. Sci. Math.


\bibitem{Kufarev}
P.P. Kufarev, {\it On one-parameter families of analytic functions} (in
Russian. English summary), Rec.~Math. [Mat. Sbornik] N.S. \textbf{13 (55)}
(1943), 87--118.

\bibitem{Kufarev1946} P.P. Kufarev, {\it On integrals of simplest differential equation with moving
 pole singularity in the right-hand side},  Tomsk. Gos. Univ. Uchyon. Zapiski, 1946, no.\,1,
 35--48.

\bibitem{Kufarev1} P.P. Kufarev, \textit{A remark on integrals of L\"owner's equation}, Doklady Akad.
Nauk SSSR (N.S.) {\bf 57} (1947), 655--656. \hbox{MR0023907 (9,421d)}

\bibitem{Kufarev_etal}
P.P. Kufarev, V.V. Sobolev\ and\ L.V. Spory\v{s}eva, \textit{A certain method of
investigation of extremal problems for functions that are univalent in the
half-plane}, Trudy Tomsk. Gos. Univ. Ser. Meh.-Mat. {\bf 200} (1968), 142--164.
\hbox{MR0257336 (41 \#1987)}

\bibitem{Kurzweil} J. Kurzweil,  Ordinary differential equations, translated from the Czech by Michal Basch,
Studies in Applied Mechanics, 13, Elsevier, Amsterdam, 1986. MR0929466
(88m:34001)


\bibitem{LawlerBook} G. F. Lawler, Conformally invariant processes in the plane, Mathematical Surveys and Monographs,
114, Amer. Math. Soc., Providence, RI, 2005. MR2129588 (2006i:60003)







\bibitem{Lind} J. Lind, \textit{A sharp condition for the Loewner equation to generate slits,} Ann.
Acad. Sci. Fenn. Math. 30 (2005), 143–158.

\bibitem{LMR} J. Lind, D.E. Marshall, and S. Rohde, \textit{Collisions and Spirals of Loewner
traces}, Duke Math. J. 154 (2010), 527-573.

\bibitem{Lind-Rohde} J. Lind and S. Rohde, Spacefilling curves and phases of the Loewner equation, preprint

\bibitem{Loewner}K. L\"{o}wner, \textit{Untersuchungen \"{u}ber schlichte
konforme Abbildungen des Einheitskreises}, Math. Ann. \textbf{89} (1923),
103--121.


\bibitem{Marshall-Rohde}D.E. Marshall and S. Rohde, \textit{The Loewner
differential equation and slit mappings}, J. Amer. Math. Soc. \textbf{18}
(2005), 763--778



\bibitem{Pommerenke-65}Ch. Pommerenke, \textit{\"{U}ber dis subordination
analytischer funktionen}, J. Reine Angew Math. \textbf{218} (1965), 159--173.

\bibitem{Pommerenke}Ch. Pommerenke, Univalent Functions.
With a chapter on quadratic differentials by Gerd Jensen, Vandenhoeck \&
Ruprecht, G\"{o}ttingen, 1975.




\bibitem{ProkhorovPreprint}D. Prokhorov, \textit{Exponential driving function for the L\"owner
equation,} Preprint 2012, arXiv:1201.0886v1 [math.CV]



\bibitem{Prokhorov-Vasiliev}D. Prokhorov\ and\ A. Vasil'ev, \textit{Singular and tangent slit solutions to the L\"owner equation}, in  Analysis
and mathematical physics, 455--463, Trends Math, Birkh\"auser, Basel.
MR2724626 (2012a:30015)


\bibitem{Schramm}O. Schramm, \textit{Scaling limits of loop-erased random
walks and uniform spanning trees}, Israel J. Math. \textbf{118} (2000),
221--288.






\bibitem{Sobolev1970} V.V. Sobolev, {\it Parametric representations for some classes
of functions univalent in half-plane}, Kemerov. Ped. Inst. Uchyon. Zapiski,
\textbf{23}(1970), 30--41.







\bibitem{Wong} C. Wong, \textit{ Smoothness of Loewner slits,} preprint.
Available on\\ \verb+http://staff.washington.edu/carto/research.html+

\end{thebibliography}
\end{document}